\documentclass[12pt]{amsart}
\usepackage{amsmath,amsfonts,amssymb,comment} 
\usepackage{amsbsy}
\usepackage{amscd}

\usepackage{amsthm}
\usepackage{bbold}
\usepackage{mathrsfs}
\usepackage{verbatim}
\usepackage [all]{xy}
\usepackage[top=1in,bottom=1in,left=1.25in,right=1.25in]{geometry}

\usepackage{tikz-cd}

\newcommand{\hide}[1]{}

\theoremstyle{plain}
\newtheorem{thm}{Theorem}[section]
\newtheorem{prop}[thm]{Proposition}

\newtheorem{lem}[thm]{Lemma}

\newtheorem{Concl}[thm]{Conclusion}

\newtheorem*{Cartan-Dieudonne}{Cartan-Dieudonne theorem {\rm (\cite [Chapter I, Theorem 7.1]{Lam})}}

\theoremstyle{definition}
\newtheorem{defi}[thm]{Definition}

\newtheorem{nota}[thm]{Notation}

\theoremstyle{remark}
\newtheorem{example}[thm]{Example}
\newtheorem{observation}[thm]{Observation}

\newtheorem{rem}[thm]{Remark}

\newcommand{\E}{{\mathcal E}}
\newcommand{\CC}{{\mathbb C}}
\newcommand{\QQ}{{\mathbb Q}}
\newcommand{\RR}{{\mathbb R}}
\newcommand{\ZZ}{{\mathbb Z}}
\newcommand{\D}{{\mathcal D}}
\newcommand{\LB}{{\mathcal L}}
\newcommand{\PP}{{\mathbb P}}
\newcommand{\Pe}{{\mathcal P}}

\newcommand{\lcm}{{\rm lcm}}

\newcommand{\Orb}{{\mathcal Orb}}

\begin{document}
\author{Nikolay Buskin}
\title[] {Every rational Hodge isometry between two $K3$ surfaces is algebraic}
\date{\today}
\maketitle
\begin{abstract}
We present a proof that cohomology classes in $H^{2,2}(S_1\times S_2)$ 
of  Hodge isometries $\psi \colon H^2(S_1,\QQ)\rightarrow H^2(S_2,\QQ)$ between any two
projective complex $K3$ surfaces $S_1$ and $S_2$
are polynomials in Chern classes of coherent sheaves.
This proves a conjecture
of Shafarevich \cite{Shaf}.
\end{abstract}

\sloppy

\tableofcontents

\section { Introduction}


A $K3$ surface is a simply connected smooth compact complex manifold of complex dimension 2
with a trivial canonical bundle. 
Let $S_1,S_2$ be $K3$ surfaces. A $K3$ lattice $\Lambda$ is an even unimodular lattice of signature (3,19). All such lattices are isomorphic.
A typical example of a $K3$ lattice is the second cohomology lattice $H^2(S, \ZZ)$ of a $K3$ surface $S$ with the bilinear
form given by the intersection from.

To any homomorphism $\varphi \colon H^2(S_1, \mathbb{Q}) \rightarrow H^2(S_2, \mathbb{Q})$ 
of Hodge structures 
we can associate a class 
$$Z_\varphi \in  H^{2,2}(S_1\times S_2, \mathbb{C})\cap H^4(S_1 \times S_2, \mathbb{Q}).$$
Here we naturally identify the vector space $H^{p,q}(S_i,\mathbb{C})^*$ with $H^{2-p,2-q}(S_i, \mathbb{C})$.
The Torelli theorem for projective $K3$ surfaces, proved by I. Shafarevich and I. Piatetski-Shapiro in \cite{Piat},
states that given an {\it effective} isometry $\varphi \colon H^2(S_1, \mathbb{Z}) \rightarrow H^2(S_2, \mathbb{Z})$
of the cohomology lattices of $K3$ surfaces $S_1$ and $S_2$, which is a homomorphism
of Hodge structures (a Hodge isometry), there exists a unique isomorphism $f\colon S_2 \rightarrow S_1$ inducing $\varphi$. 
So, in particular, the cycle $Z_{\varphi}$, cohomologous to the graph $\Gamma_f$ of the map $f$, is algebraic.
In his 1970 ICM talk 
\cite{Shaf}  I.  Shafarevich  asked if $Z_{\varphi}$ is algebraic
for any {\it rational} Hodge isometry $\varphi\colon  H^2(S_1, \mathbb{Q}) \rightarrow H^2(S_2, \mathbb{Q})$. 
Our goal is to prove the following theorem.

\begin{thm}
\label{Main-Theorem-Algebraic} Let $S_1,S_2$ be projective $K3$ surfaces. The class $Z_\varphi$ of any Hodge isometry 
$\varphi \colon  H^2(S_1, \QQ) \rightarrow H^2(S_2, \QQ)$
is algebraic.
\end{thm}
This theorem is proved in Section \ref{Reduction-to-the-case-of-a-cyclic-isometry}.

S. Mukai's groundbreaking work \cite{Mukai1} made a significant progress towards this result
(see  Section 
\ref{Mukai's result: classical and new formulation} below). His main result implies Theorem \ref{Main-Theorem-Algebraic}
in the case when the Picard rank of the surfaces $S_1$ and $S_2$  is greater or equal 11
(Hodge isometric surfaces obviously have equal Picard rank),  see
\cite[Cor. 1.10]{Mukai1}.
Nikulin \cite{Nikulin2} extended Mukai's argument to the case of projective elliptic surfaces,
which by Meyer's Theorem \cite[Cor. 5.10]{Gerstein} implies  Theorem \ref{Main-Theorem-Algebraic} for surfaces with Picard number greater or equal 5.
Mukai announced also Theorem \ref{Main-Theorem-Algebraic} in his 2002  ICM talk \cite{Mukai3}.

%
Let us give now a brief lattice-theoretic introduction
which is needed for understanding the outline of our strategy and then give the outline.
Let us fix a $K3$ lattice $\Lambda$.
Further we will need the notion of a {\it marking} of a $K3$ surface $S$, that is a lattice isometry $\eta_S: H^2(S, \ZZ) \rightarrow \Lambda$.
Let us choose markings $\eta_1, \eta_2$ for the surfaces $S_1, S_2$. For the rest of the introduction 
the notation  $\varphi$ stands for a rational Hodge isometry $\varphi: H^2(S_1,\QQ) \rightarrow H^2(S_2,\QQ)$
and $\phi$ stands for the rational  isometry $\phi=\eta_2\circ \varphi \circ \eta^{-1}_1: \Lambda_\QQ \rightarrow \Lambda_\QQ$, where $\Lambda_\QQ=\Lambda\otimes \QQ$. So, the markings allow us to assign to rational isometries
of different surfaces some elements of the group of isometries of $\Lambda_\QQ$, for which there exists a
nice structure theory. The rational isometries of $\Lambda_\QQ$ induced by (quasi-)universal sheaves are of {\it cyclic} type which
means that it is of the form $g\circ r_x \circ h$, where $g,h$ are integral isometries of $\Lambda$
and $r_x$ is a reflection of $\Lambda_\QQ$, $v \mapsto v-\frac{2(v,x)}{(x,x)}x$ for 
a non-isotropic primitive $x \in \Lambda$. Such a reflection is, in general, defined over $\QQ$, being integral precisely when
$(x,x)=\pm 2$, for any even unimodular lattice $\Lambda$.
We say that the above isometry $g\circ r_x \circ h$ is of 
{\it $n$-cyclic type}\footnote{In Definition \ref{Cyclic-type-definition} a different definition of $n$-cyclic type will be used, in terms of a criterion which is easier to verify. Both definitions are equivalent by Proposition \ref{prop-only-one-double-orbit}.}, if
the primitive vector $x$ satisfies $|(x,x)|=2n$.
We will also say that $\varphi$ is of cyclic type if $\phi=\eta_2\varphi\eta^{-1}_1$ is of cyclic type.
It is immediately clear that this definition does not depend on the choice of markings $\eta_1,\eta_2$.
In Section \ref{Double-orbits}  we prove a lattice-theoretic criterion for an isometry to be of $n$-cyclic type.
%
%

First, we prove Theorem \ref{Main-Theorem-Algebraic} in the particular case when the Hodge isometry $\varphi \colon  H^2(S_1,\QQ) \rightarrow H^2(S_2,\QQ)$ is of {\it cyclic type}.
%
For a general rational Hodge isometry $\varphi$, we have a reduction to the cyclic type case. Namely, as will be shown later, the group of rational isometries of $\Lambda_\QQ$ is generated by isometries of cyclic type.
Representing a general rational isometry $\phi \colon \Lambda_\QQ \rightarrow \Lambda_\QQ$ as a composition of cyclic ones,
 $\phi = \phi_k \circ \dots \circ\phi_1$, we then find a sequence of surfaces $S^{\prime}_1=S_1, \dots, S^{\prime}_{k+1}=S_2$ together with their markings $\eta_i$ so that 1) for each pair $S^{\prime}_i, S^{\prime}_{i+1}$ the isometry $\varphi_i = \eta^{-1}_{i+1}\phi_i\eta_{i}$ is Hodge; 2) each $\varphi_i$ is induced by an algebraic class  on 
$S^{'}_i \times S^{'}_{i+1}$. Now taking composition of algebraic classes as correspondences we get our $\varphi$ represented by an algebraic class.
This will prove Theorem \ref{Main-Theorem-Algebraic}. 

In order to prove that every rational Hodge isometry $\varphi$ of cyclic type is algebraic we need the following four ingredients:

1) For every $n \geqslant 2$ we need an example of  a $K3$ surface $S$, a two-dimensional smooth and projective moduli space  $M$ of stable vector
bundles on $S$ and a  universal (untwisted) sheaf $\mathcal E$ over $S \times M$
inducing via the $\kappa$-class a rational Hodge isometry of $n$-cyclic type  from $H^2(S,\QQ)$ to $H^2(M,\QQ)$,
see Conclusion \ref{Existence-of-n-type-isometry-for-any-n};

2) Fix a rational isometry $\phi \colon  \Lambda_\QQ \rightarrow \Lambda_\QQ$ of $n$-cyclic type for some $n \geqslant 2$.  
Consider the locus in the moduli space of marked pairs $((S_1,\eta_1),(S_2,\eta_2))$ of $K3$ surfaces 
along which $\varphi=\eta^{-1}_2 \phi\eta_1$ stays of Hodge type.
We need to show that this locus is covered by twistor lines analogous to twistor lines in the moduli space 
of marked $K3$
surfaces. The proof of this fact
is given in Proposition \ref{prop-connected-components};

3) We observe that any two {\it signed} isometries of $n$-cyclic type belong to the locus introduced in 2)
and are, thus,  deformation equivalent (this is 
proved by Propositions 
\ref{prop-only-one-double-orbit}
and \ref{prop-connected-components});

4) M. Verbitsky's result on hyperholomorphic sheaves implying  that for every example in 1) the universal sheaf $\mathcal E$ over $S \times M$ can be extended as a twisted sheaf over a twistor family containing the hyperk\"ahler manifold $S\times M$, see Theorem \ref {Hyper-theorem} and the discussion after it. A repeated
application of this result will enable us to deform $\mathcal E$ to a sheaf on the product $S_1 \times S_2$
for any $((S_1,\eta_1),(S_2,\eta_2))$ in the corresponding locus as in 2).

One of the applications of the main result of this work is 
the following corollary:
\vspace*{0.1cm}

\noindent {\bf Corollary.} {\it If $S$ is an algebraic $K3$ surface for which the endomorphism field of
its transcendental Hodge structure is a CM-field then the Hodge conjecture
is true for $S\times S$.}

Indeed, by \cite[Lemma 5.3]{RM} and \cite[Sec. 1.5]{Zarhin} when the endomorphism field of $T(S)_\QQ$
is a CM-field, it is spanned over $\QQ$ by Hodge isometries of $T(S)_\QQ$. Thus,
by Theorem \ref{Main-Theorem-Algebraic} every Hodge endomorphism of $T(S)_\QQ$
is algebraic.
%

\section {Mukai's result: classical and new formulations}
\label{Mukai's result: classical and new formulation}

Here we recall the classical formulation of Mukai's result and reformulate it in the form that
we will use later in the proof of the main theorem.

Let $S$ be an algebraic $K3$ surface.
In \cite{Mukai1} Mukai introduces a weight 2 Hodge structure on the (commutative) cohomology ring $H^*(S, \CC)$,
$$\widetilde{H}^{2,0}(S, \CC)=H^{2,0}(S,\CC),$$
$$\widetilde{H}^{0,2}(S, \CC)=H^{0,2}(S,\CC),$$
$$\widetilde{H}^{1,1}(S, \CC)=H^0(S,\CC) \oplus H^{1,1}(S,\CC)  \oplus H^4(S,\CC).$$
Given a class $\alpha \in H^*(S, \ZZ)$ denote by $\alpha^i$ its graded summand in $H^i(S,\ZZ)$.
Denote by $\widetilde{H}(S,\ZZ)$ the lattice $H^*(S,\mathbb Z)$  with the integral bilinear Mukai pairing $$(\alpha, \beta)=
-\alpha^4\beta^0+\alpha^2\beta^2-\alpha^0\beta^4 \in H^4(S, \ZZ) \cong \ZZ.$$
Now let
$$v =(n, \alpha, s) \in \widetilde{H}^{1,1}(S, \mathbb{Z})=H^{0}(S, \mathbb{Z})\oplus H^{1,1}(S, \mathbb{Z}) \oplus H^{4}(S, \mathbb{Z})$$ be a vector isotropic  with respect to Mukai pairing and let $h$ be an ample divisor on $S$. Recall that $H^{1,1}(S, \mathbb{Z})$ is, by definition, the intersection $H^{1,1}(S) \cap H^2(S,\ZZ)$
and assume that there exists a compact moduli space $M=M_h(v)$ of vector bundles $E$ on $S$, slope-stable with respect to $h$ ($h$-slope-stable), whose Mukai vector $v(E) = ch(E)\sqrt{td_S}$ is equal to $v$. 
%
%
For a 
coherent sheaf $\mathcal E$ on $S\times M$ there is a class 
$$Z_{\mathcal E}=ch({\mathcal E})\sqrt{td_{S\times M}},$$
which determines a homomorphism of Hodge structures
$$f_{\mathcal E} \colon  \widetilde{H}(S, \mathbb{Q}) \rightarrow \widetilde{H}(M, \mathbb{Q}),$$
$$\alpha \mapsto {\pi_M}_*(\pi_S^*(\alpha)\cdot Z_{\mathcal E}).$$
Above, ``$\cdot$'' is the standard multiplication in the cohomology ring $H^*(S, \mathbb{Q})$.

We recall  that for an algebraic $K3$ surface $S$ its transcendental lattice $T(S)$ is defined as the integral lattice $(H^{1,1}(S, \RR)\cap H^2(S,\ZZ))^{\perp}$, where the orthogonal complement is taken in $H^2(S, \ZZ)$ with respect to the intersection form. 
Note that for $v$ as above $T(S)$ is contained in the orthogonal complement $v^{\perp} \subset \widetilde{H}(S,\ZZ)$ and $T(S) \cap \QQ v =\{0\}$. 
We denote by $T(S)_\QQ$ the $\QQ$-vector space $T(S)\otimes_{\ZZ}\QQ$.
For a locally free
sheaf $\mathcal E$ we denote by $\mathcal E^{\vee}$ its dual sheaf.
\begin{thm}{\rm (\cite [Thm. 1.5, Thm. A.5, Thm. 4.9, Prop. 6.4]{Mukai1})} \label{MukaiTheorem1} 
There exists a locally free
sheaf $\mathcal E$ on $S \times M$
and a positive integer $\sigma(\mathcal E)$ 
such that 
$$\frac{1}{\sigma(\mathcal E)}\widetilde{f}_{\mathcal E^{\vee}} \colon  v^{\perp}/\ZZ v 
\xrightarrow{\sim} H^2(M, \ZZ) \subset \widetilde{H}(M,\ZZ)$$
is a Hodge isometry. This isometry is independent of the choice of the 'quasiuniversal' sheaf $\mathcal E$.
The restriction 
\begin{equation}
\label{transcendental-Hodge-isometry}
\frac{1}{\sigma(\mathcal E)}\widetilde{f}_{\mathcal E^{\vee}}|_{T(S)} \colon T(S) \rightarrow T(M)
\end{equation} 
is an isometric embedding with cokernel a finite cyclic group. 
In the case when the sheaf $\mathcal E$ is universal ($\sigma (\mathcal E)=1$) already 
the map $f_{\mathcal E^\vee} \colon \widetilde{H}(S,\ZZ) \rightarrow \widetilde{H}(M,\ZZ)$
is a Hodge isometry.
\end{thm}
Here the Hodge structure on $v^{\perp}/\ZZ v$ is induced from that of 
$ \widetilde{H}(S, \ZZ)$. Mukai constructs the sheaf $\mathcal E$ as a 'quasiuniversal sheaf',
see \cite{Mukai1} for the terminology. 

\subsection {A new formulation of Mukai's result.}

\label {A new formulation of Mukai's result}

Generally speaking, the new formulation is obtained by canonically extending the above Mukai's Hodge isometry
$\frac{1}{\sigma(\mathcal E)}\widetilde{f}_{\mathcal E^{\vee}}|_{T(S)_\QQ} \colon T(S)_\QQ \rightarrow T(M)_\QQ$ 
to a rational Hodge isometry from $H^2(S, \QQ)$ to $H^2(M,\QQ)$. 
In the view of independence stated in \ref{MukaiTheorem1} 'canonically'
means that we are extending it by using a certain characteristic class of $\mathcal E$.
The language of quasiuniversal sheaves used by Mukai in \cite{Mukai1} is not appropriate for us, as sheaves in general do not stay untwisted
under deformations. We need to consider deformations  of sheaves, so we need 
a more general setup of {\it twisted sheaves}. For an account of this theory we refer
to A. Caldararu's PhD Thesis \cite{Cald}.
Recall the definition of a twisted sheaf on a complex manifold $X$. Let $\alpha \in C^2(X, \mathcal O^*_X)$ be a \v{C}ech 2-cocycle in the analytic topology,
determined by an open cover $\mathcal U=\{U_i\}_{i\in I}$ of $X$ and sections $\alpha_{ijk}\in \Gamma (U_i\cap U_j \cap U_k, \mathcal O^*_X)$. 
We define an $\alpha$-twisted sheaf on $X$ to be a pair $(\{\mathcal F_i\}_{i \in I}, \{\theta_{ij}\}_{i,j \in I})$
where $\mathcal F_i$ are sheaves of coherent $\mathcal O_X$-modules on $U_i$ and $\theta_{ij} \colon \mathcal F_j|_{U_i \cap U_j} \rightarrow \mathcal F_i|_{U_i \cap U_j}$ are isomorphisms
such that:

1) $\theta_{ii}$ is the identity for all $i \in I$;

2) $\theta_{ij}=\theta^{-1}_{ji}$ for all $i,j \in I$;

3) $ \theta_{ij}\circ \theta_{jk} \circ \theta_{ki}$ is the multiplication
by $\alpha_{ijk}$ on $\mathcal F_i|_{U_i\cap U_j \cap U_k}$ for all $i,j,k \in I$.

\vspace*{3mm}

We want to formulate Mukai's result in terms of a Hodge homomorphism
induced by a characteristic class defined for twisted torsion free sheaves of positive rank.
Let $\mathcal E$ be a twisted locally free sheaf of rank $n>0$ over a compact complex manifold $X$.  Then the sheaf $\mathcal E^{\otimes n} \otimes det(\mathcal E^{\vee})$ is untwisted.
Now we set $$\kappa(\mathcal E)=(ch(\mathcal E^{\otimes n} \otimes det(\mathcal E^{\vee})))^{\frac{1}{n}},$$
where $(n^n+x)^{\frac{1}{n}}$ is equal to $n+\frac{1}{n^n}(x-n^n)+\dots$, the Taylor series of the $n$-th root function centered at $n^n$. 
It is easy to see that when $\mathcal E$ is an untwisted sheaf, the class $\kappa(\mathcal E)$ is equal to $ch(\mathcal E)\cdot e^{-\frac{c_1(\mathcal E)}{n}}$. For a (twisted or untwisted) sheaf $\mathcal E$  over $X=S \times M$ the class $\kappa({\mathcal E^{\vee}})\sqrt{td_{S\times M}}$ induces a 
map $\psi_{\mathcal E} \colon H^*(S,\QQ)\rightarrow H^*(M,\QQ)$
in the same fashion as the above defined $f_{\mathcal E^\vee}$ (for brevity of notation we omit '$\vee$' in the subscript of $\psi_{\mathcal E}$). 
The map $\psi_{\mathcal E}$ is a rational homomorphism of Hodge structures $\widetilde{H}(S,\QQ)$ and $\widetilde{H}(M,\QQ)$. 

\begin{lem} 
For a universal untwisted sheaf  $\mathcal E$ over $S\times M$ the map 
$\psi_{\mathcal E} \colon \widetilde{H}(S,\QQ) \rightarrow \widetilde{H}(M,\QQ)$ 
is a Hodge isometry which restricts
to a Hodge isometry $\psi_{\mathcal E}|_{H^2(S,\QQ)} \colon H^2(S, \QQ) \rightarrow H^2(M, \QQ)$.
The restriction $\psi_{\mathcal E}|_{T(S)_\QQ}$ is equal to the restriction $\widetilde{f}_{\mathcal E^\vee}|_{T(S)_\QQ}$.
\label {Reformulation}
\end{lem}
\begin{proof}

Note that for an untwisted $\mathcal E$ due to equality 
$$\kappa(\mathcal E^\vee)=ch(\mathcal E^\vee)\cdot e^{-\frac{c_1(\mathcal E^\vee)}{n}}=
\pi_S^*e^{\frac{c_1(\mathcal E|_{S\times \{m\}})}{n}} \cdot ch(\mathcal E^\vee) \cdot \pi_M^*e^{\frac{c_1(\mathcal E|_{\{s\}\times M})}{n}}$$ we have $$\psi_{\mathcal E}=
e^\beta \circ f_{\mathcal E^\vee} \circ e^\alpha,$$
where $\beta={\frac{c_1(\mathcal E|_{\{s\}\times M})}{n}}$ and $\alpha={\frac{c_1(\mathcal E|_{S\times \{m\}})}{n}}$
and the maps $e^\alpha \colon \widetilde{H}(S,\QQ) \rightarrow \widetilde{H}(S,\QQ), 
e^\beta  \colon \widetilde{H}(M,\QQ) \rightarrow \widetilde{H}(M,\QQ)$ are considered as multiplications by the corresponding classes. As $\alpha, \beta$ are classes of type $(1,1)$, the maps $e^\alpha, e^\beta$ are certainly Hodge isometries of 
$\widetilde{H}(S,\QQ), \widetilde{H}(M,\QQ)$.
The map $f_{\mathcal E^\vee}$ is a Hodge isometry due to Theorem \ref{MukaiTheorem1}.
Thus, $\psi_{\mathcal E}$ is indeed a Hodge isometry.

We need to show that $\psi_{\mathcal E}$ takes $H^2(S,\QQ)$ to $H^2(M,\QQ)$.
For this it is sufficient to show that the isometry $\psi_{\mathcal E}$ maps $H^2(S,\QQ)^{\perp}=H^0(S,\QQ)\oplus H^4(S,\QQ)$,
to $H^2(M,\QQ)^\perp=H^0(M,\QQ)\oplus H^4(M,\QQ)$.
Let $\Phi^{S \rightarrow M}_{\mathcal E^\vee} \colon D^b(S) \rightarrow D^b(M)$ 
be the Fourier-Mukai functor with kernel $\mathcal E^\vee$ 
and let $v$ be the Mukai vector map, $v:D^b(X)\rightarrow H^*(X,\ZZ)$, $X=S$ or $M$.
We have the following commutative diagram (see \cite[Cor. 5.29]{Huyb})
\[
\xymatrix{
D^b(S) \ar[r]^{\Phi^{S \rightarrow M}_{\E^\vee}} \ar[d]_{v} &
D^b(M) \ar[d]^{v}
\\
\widetilde{H}(S,\QQ) \ar[r]_{f_{\E^\vee}} &
\widetilde{H}(M,\QQ). 
}
\]

As was proved in \cite{Mukai1} (although not formulated explicitly), 
for a universal $\mathcal E$ the Fourier-Mukai functor
$\Phi^{S \rightarrow M}_{\mathcal E^\vee} \colon D^b(S) \rightarrow D^b(M)$ 
induces a derived equivalence between bounded derived categories $D^b(S)$
and $D^b(M)$.

For the sky-scraper sheaf $\CC_p$ over $S$ its Mukai vector
$v(\CC_p)=(0,0,1) \in H^4(S,\QQ)$ and the Fourier-Mukai transform $\Phi_{\mathcal E^\vee}(\CC_p)=\mathcal E^\vee|_{\{s\}\times M}$, we have that
$\psi_{\mathcal E}((0,0,1))=\psi_{\mathcal E}(v(\CC_p))=e^\beta \circ f_{\mathcal E^\vee}(e^\alpha\cdot v(\CC_p))
=e^\beta \cdot f_{\mathcal E^\vee}(v(\CC_p))=e^\beta(v(\Phi_{\mathcal E^\vee}(\CC_p)))=
e^\beta \cdot v(\mathcal E^\vee_{\{s\}\times M})) = (n,0,s)\in H^0(M,\QQ) \oplus H^4(M,\QQ)$ where $n>0$.

Now 
as $(0,0,1)$ is isotropic, $\psi_{\mathcal E}((0,0,1))$ must be isotropic as well,
so that $s=0$ and thus $\psi_{\mathcal E}((0,0,1))=(n,0,0)$.
The map $f_{\mathcal E^\vee}$ is adjoint to $f_{\mathcal E^\vee}^{-1}$ with respect
to the Mukai pairing, and  $f_{\mathcal E^\vee}^{-1}$ is the map induced by the functor $\Phi^{M \rightarrow S}_{\mathcal E[2]}$
 on the cohomological level (this follows from \cite[Prop. 5.9, Lemma 5.32]{Huyb}, the details are explained
in the proof of \cite[Prop. 5.44]{Huyb}). So repeating the above argument
for the isometry $\psi_{\mathcal E}^{-1}=e^{-\alpha} \circ  f_{\mathcal E^\vee}^{-1} \circ e^{-\beta} \colon  \widetilde{H}(M,\QQ) \rightarrow \widetilde{H}(S,\QQ)$
we see that $\psi_{\mathcal E}^{-1}(0,0,1)=(n,0,0)$, thus $\psi_{\mathcal E}(1,0,0)=(0,0,\frac{1}{n})$. 
This shows that $\psi_{\mathcal E}$ maps $H^2(S,\QQ)^{\perp}$ to $H^2(M,\QQ)^{\perp}$.

The fact that $\psi_{\mathcal E}|_{T(S)_\QQ}=\widetilde{f}_{\mathcal E^\vee}|_{T(S)_\QQ}$ follows from
the fact that both maps $\psi_{\mathcal E}|_{T(S)_\QQ}$ and $f_{\mathcal E^\vee}|_{T(S)_\QQ}$ 
are degree preserving and thus are actually induced by degree 4 components $\kappa_2(\mathcal E^\vee)$ and 
$ch_2(\mathcal E^\vee)$. Now, the relation $\kappa(\mathcal E^\vee)=ch(\mathcal E^\vee)e^{-\frac{c_1(\mathcal E^\vee)}{n}}$ tells us that $\kappa_2(\mathcal E^\vee)=ch_2(\mathcal E^\vee)-\frac{c_1(\mathcal E^\vee)^2}{2n}$.
As the class $\frac{c_1(\mathcal E^\vee)^2}{2n} \in H^{1,1}(S)\otimes H^{1,1}(M)$ induces 
zero map on $T(S)_\QQ$, we have that indeed $\psi_{\mathcal E}|_{T(S)_\QQ}=\widetilde{f}_{\mathcal E^\vee}|_{T(S)_\QQ}$.
This completes the proof of the lemma.
\end{proof}

\section {Double orbits}
\label {Double-orbits}
All the constructions and proofs in this section
are due to Eyal Markman.
Let $\Lambda$ be the $K3$ lattice.
In this section we show that for every $n \in \ZZ^+$ there exists only one double orbit $O(\Lambda)\phi O(\Lambda) \subset O(\Lambda_\QQ)$ where $\phi$ is a rational isometry  of $n$-cyclic type, a notion which will be introduced below.
Here $O(\Lambda)$ and $O(\Lambda_\QQ)$ denote the group of integral isometries of $\Lambda$ and the group
of rational isometries of $\Lambda_\QQ$, respectively.


{\bf Notation:}
Given a lattice $L$, denote
the pairing on $L$ by $(\cdot,\cdot)$ 
 and denote by $O(L)$ and $O(L_\QQ)$ the corresponding integral and rational isometry groups.
Given a sub-lattice $M$ of $L$ of finite index, we regard $M^*$  as the subgroup of $L_\QQ$
of elements $\lambda$, such that $(\lambda,m)$ is an integer for every element $m\in M$.
We then have a flag $M\subset L\subset L^*\subset M^*$.

Given $\phi\in O(L_\QQ)$, set 
\[
I_\phi :=L\cap \phi^{-1}(L).
\]
\begin {defi}
\label{Cyclic-type-definition}
We say that $\phi$ is of {\em $n$-cyclic type}, if $L/I_\phi$ is a cyclic group of order $n$.
\end{defi}
\begin{example}
Let $x$ be a primitive element of an even unimodular lattice $L$
satisfying $(x,x)=2d$, for a non-zero integer $d$. Let $r_x \colon L_\QQ \rightarrow L_\QQ$
be the reflection as in Introduction.
Then $r_x$ is of $|d|$-cyclic type, since $I_{r_x}$ is the preimage of the ideal $(d) \subset \ZZ$
via the surjective homomorphism $(x,\cdot) \colon L \rightarrow \ZZ$.
\end{example}
The {\em double orbit} of an element $\phi\in O(L_\QQ)$ is the subset 
$O(L)\phi O(L)$. 
The set of double orbits $[\phi]=O(L)\phi O(L)$ will be denoted by $O(L)\backslash O(L_\QQ)/O(L)$.
Let us now formulate the most important result of this section.

\begin{prop}
\label {prop-only-one-double-orbit}
{\it Let $\phi_1$ and $\phi_2$ be rational isometries of $\Lambda_\QQ$ of $n$-cyclic type (in the sense of Definition \ref{Cyclic-type-definition}).
Then $$[\phi_1]=[\phi_2].$$}
\end{prop}

The proof of this proposition will be given in Subsection \ref {Double-orbits-in-the-$K3$-lattice}.
We suggest to skip Subsections 
\ref{Double-orbits-in-the-even-rank-two-unimodular-hyperbolic-lattice} and
\ref {Double-orbits-in-the-$K3$-lattice} on a first reading, 
as the proofs there involve lattice-theoretic techniques not used elsewhere in the paper.
We would recommend not to skip Subsection \ref{An-important-example} which 
realizes step 1) in our strategy sketch in the introduction.

\subsection{Double orbits in the even rank two unimodular hyperbolic lattice}
\label {Double-orbits-in-the-even-rank-two-unimodular-hyperbolic-lattice}
Let $U$ be the rank $2$ even unimodular lattice of signature $(1,1)$. 
The proof of the following lemma is elementary and is left to the reader. 
\begin{lem}
\label{lemma-double-orbits-for-U}
Any rational isometry of $U_\QQ$ is of cyclic type.
There exists a one to one correspondence between double orbits of
rational isometries of $U_\QQ$ of $n$-cyclic type with $n>1$
and pairs $(a,b)$ of integers satisfying $a>b>0$, $\gcd(a,b)=1$, and $ab=n$.
\end{lem}

Let $f$ be the element of $O(U_\QQ)$ such that $f(e_1)=\frac{a}{b}e_1$, 
$f(e_2)=\frac{b}{a}e_2$, where $ab=n$ and $\gcd(a,b)=1$. 
Note that $I_f^*/I_f$ is isomorphic to $(\ZZ/n\ZZ)\times (\ZZ/n\ZZ)$, 
with generators $\left\{\frac{e_1}{a}+I_f, \frac{e_2}{b}+I_f\right\}$.
Let 
\[
q  \colon  I_f^*/I_f \ \rightarrow \QQ/2\ZZ
\]
be the residual quadratic form. 
A subgroup of $I_f^*/I_f$ 
is called {\em isotropic} if $q$ vanishes on it. A subgroup of $I_f^*/I_f$ 
is called {\em lagrangian}, if it is an isotropic cyclic group of order $n$.

Let $\D_n$ be the set of all ordered pairs of positive integers $(c,d)$, satisfying $n=cd$ and $\gcd(c,d)=1$.
Given a pair $(c,d)$ in $\D_n$, let 
\[
L_{(c,d)} := \left\langle \left(\frac{c}{a}\right)e_1+I_f, \left(\frac{d}{b}\right)e_2+I_f\right\rangle
\]
be the subgroup of $I_f^*/I_f$ generated by the two cosets above.
Note that as orders of the generators $\left(\frac{c}{a}\right)e_1+I_f, \left(\frac{d}{b}\right)e_2+I_f$
in $L_{(c,d)}$
are $d$ and $c$, respectively, and thus are coprime, the group $L_{(c,d)}$ is
indeed cyclic.
Note the equalities 
\[
L_{(a,b)}=U/I_f \ \ \ \mbox{and} \ \ \  L_{(b,a)}=f^{-1}(U)/I_f.
\]
Let $\LB_n$ be the set of lagrangian subgroups of $I_f^*/I_f$. 

\begin{lem}
\label{lemma-classification-of-lagrangian-subgroups}
\begin{enumerate}
\itemс
\label{lemma-item-one-to-one}
The map $(c,d) \mapsto L_{(c,d)}$ is a one-to-one correspondence between the sets 
$\D_n$ and $\LB_n$.
\item
\label{lemma-item-order-of-L-c-d}
$L_{(c_1,d_1)}\cap L_{(c_2,d_2)}=
\langle\lcm(c_1,c_2)\frac{e_1}{a}+I_f,
\lcm(d_1,d_2)\frac{e_2}{b}+I_f\rangle$.
\item
\label{lemma-item-complementary-lagrangian-subgroups}
$L_{(c_1,d_1)}\cap L_{(c_2,d_2)}=(0)$ if and only if $(c_1,d_1)=(d_2,c_2)$. 
\end{enumerate}
\end{lem}

\begin{proof}
(\ref{lemma-item-one-to-one}) 
Injectivity is clear. We prove only the surjectivity.
Let $L$ be a lagrangian subgroup of $I_f^*/I_f$ and $(x\frac{e_1}{a},y\frac{e_2}{b})$ a generator of $L$.
The value $2xy/n$  of $q$ on this generator is congruent to $0$ modulo $2\ZZ$.
Hence, $n$ divides $xy$. Set $c:=\gcd(x,n)$ and $d:=\gcd(y,n)$. Then $n$ divides $cd$
and $L$ maps injectively into the product of the direct sum of the two cyclic subgroups  
$C_1$ generated by $x\frac{e_1}{a}+I_f$ and $C_2$ generated by $y\frac{e_2}{b}+I_f$. 
The order of $C_1$ is $n/c$ and
the order of $C_2$ is $n/d$. Hence, the order $n$ of $L$ divides the order $n^2/(cd)$ of
$C_1\times C_2$. We conclude that 
$n=cd$ and $L$ is isomorphic to the product of $C_1$ and $C_2$. It follows that $C_1\times C_2$
is cyclic and so the orders of $C_1$ and $C_2$ are  coprime. Hence, $\D_n$ maps surjectively onto $\LB_n$.

(\ref{lemma-item-order-of-L-c-d}) is clear. 

(\ref{lemma-item-complementary-lagrangian-subgroups}) 
The ``if'' direction follows immediately from part (\ref{lemma-item-order-of-L-c-d}). We prove the ``only if'' direction.
$L_{(c,d)}\cap L_{(\tilde{c},\tilde{d})}=(0)$ if and only if 
$\lcm(c_1,c_2)=n$ and $\lcm(d_1,d_2)=n$, if and only if
$\frac{c_1c_2}{\gcd(c_1,c_2)}=n$ and $\frac{d_1d_2}{\gcd(d_1,d_2)}=n$.
Now $c_1c_2d_1d_2=n^2$. Hence, $\gcd(c_1,c_2)=1=\gcd(d_1,d_2)$, $c_1c_2=n$, and $d_1d_2=n$.
The equality $(c_1,d_1)=(d_2,c_2)$ follows.
\end{proof}

\subsection{Double orbits in the $K3$ lattice}
\label{Double-orbits-in-the-$K3$-lattice}

We prove Proposition \ref {prop-only-one-double-orbit} in this section. 


Let $\Lambda$ be the $K3$ lattice. 
Fix a primitive isometric embedding of $U$ in $\Lambda$. We get the orthogonal decomposition 
$\Lambda=U\oplus U^\perp$. 
Given a rational isometry 
$f\in O(U_\QQ)$ we get a rational isometry $\widetilde{f}$ in $O(\Lambda_\QQ)$ by extending $f$
as the identity on $U^\perp$. Let
\[
\epsilon \  \colon  \ O(U)\backslash O(U_\QQ)/O(U) \ \ \ \longrightarrow \ \ \ 
O(\Lambda)\backslash O(\Lambda_\QQ)/O(\Lambda)
\]
be the function sending $[f]$ to $[\widetilde{f}]$.

For a general lattice $L$ the subset of double orbits in $O(L)\backslash O(L_\QQ)/O(L)$ 
of cyclic type of order $n$ will be denoted by 
\[
\Orb(L,n) \ \ \ \subset \ \ \ O(L)\backslash O(L_\QQ)/O(L).
\]

\begin{lem}
\label{lemma-epsilon-is-surjective}
The function $\epsilon$ restricts to a surjective map
\[
\epsilon \  \colon  \ \Orb(U,n) \ \ \ \rightarrow \ \ \ \Orb(\Lambda,n)
\]
for every integer $n$.
\end{lem}

\begin{proof}
Let $\phi$ be a rational isometry of $\Lambda_\QQ$ of $n$-cyclic type. Then 
both $\Lambda/I_\phi$ and $I_\phi^*/\Lambda$ are cyclic of order $n$.
Choose an element $x\in I_\phi^*\subset \Lambda_\QQ$, which maps to a generator of $I_\phi^*/\Lambda$.
We have
\begin{equation}
\label{eq-I-phi-in-terms-of-x}
I_\phi=\{\lambda\in\Lambda \ | \ (x,\lambda) \ \mbox{is an integer}\}.
\end{equation}
The element $y:=nx$ belongs to $\Lambda$. Set $\mbox{div}(y,\cdot):=\gcd\{(y,\lambda) \ | \lambda\in \Lambda\}$.
Then $n$ and $\mbox{div}(y,\cdot)$ are relatively prime. We may assume that $y$ is a primitive
element of $\Lambda$,
possibly after 
adding to $x$ a primitive element $x_1$ of $\Lambda$.

The isometry group of $\Lambda$ acts transitively on the set of primitive elements of self-intersection $2d$,
for any integer $d$, see, for example, \cite[\S 6, appendix]{Piat}. In other words, this set consists of a single $O(\Lambda)$-orbit. 
The element $e_1+de_2$ of $U$ has degree $2d$. Hence, $U$ contains primitive elements of any even degree
and there exists an isometry $g\in O(\Lambda)$, such that $g(y)$ belongs to $U$. 

Observe that $g^{-1}(U^\perp)$ is contained in $I_\phi$ by using
that for any element $t\in U^\perp$, we have
\[
(x,g^{-1}(t))=(g(x),t)=(g(y),t)/n=0.
\]
Now Equality (\ref{eq-I-phi-in-terms-of-x}) implies that $g^{-1}(U^\perp) \subset I_\phi$.

The lattice $U^\perp$ is unimodular, and so $\phi(g^{-1}(U^\perp))$ is a unimodular sublattice of $\Lambda$.
The orthogonal complement 
\[ T := 
[\phi(g^{-1}(U^\perp))]^\perp
\] 
of the latter in $\Lambda$ is thus isometric to $U$. There exists a unique $O(\Lambda)$-orbit of
isometric embeddings of $U$ into $\Lambda$, see, for example, \cite[\S 6, appendix]{Piat}. Hence, there exists an isometry $h\in O(\Lambda)$,
such that $h(T)=U$. Set
\[
\psi \ := \ h\phi g^{-1}.
\]
Then $\psi$ leaves $U^\perp$ invariant and restricts to $U^\perp$ as an integral isometry $\psi_2$.
It follows that $\psi$ leaves $U_\QQ$ invariant and restricts to $U_\QQ$ as a rational isometry $f$.
Let $\widetilde{\psi}_2\in O(\Lambda)$ be the extension of $\psi_2$ via the identity on $U$. The extension
$\widetilde{f}$ 
of $f$ to an element of $O(\Lambda_\QQ)$ via the identity on $U^\perp$ 
satisfies the equality
\[
\widetilde{f} \ \ = \ \ 
\widetilde{\psi}_2^{-1}h\phi g^{-1}.
\]
Hence, the double orbit $[\phi]$ is equal to $\epsilon([f])$.
\end{proof}

\begin{lem}
\label{lemma-intersection-sublattice-determine-the-double-orbit}
Let $\phi_1$ and $\phi_2$ be rational isometries of $\Lambda_\QQ$ of $n$-cyclic type.
Assume that $I_{\phi_1}=I_{\phi_2}$. Then $O(\Lambda)\phi_1=O(\Lambda)\phi_2$. In particular, 
$[\phi_1]=[\phi_2]$.
\end{lem}

\begin{proof}
Set $I:=I_{\phi_1}$ and $I^*:=I_{\phi_1}^*$. Then $I^*=I_{\phi_2}^*$ as well.
There exist integral isometries $g_i$ and $h_i$ in $O(\Lambda)$,
and a rational isometry $f_i$ in $O(U_\QQ)$, such that 
$\phi_i=g_i\widetilde{f}_ih_i$, by  the surjectivity of $\epsilon$. 
Then $I=I_{\phi_i}=h_i^{-1}(I_{g_i\tilde{f}_i})=h_i^{-1}(I_{\tilde{f}_i})$. 
Hence, the isometry $h_i$ descends to an isometry from 
$I^*/I$ onto $I^*_{\tilde{f}_i}/I_{\tilde{f}_i}$ with respect to their respective residual quadratic forms.
Furthermore, $h_i(\Lambda/I)=\Lambda/I_{\tilde{f}_i}$. Note also that $I^*_{\tilde{f}_i}/I_{\tilde{f}_i}$
is naturally isomorphic to $I^*_{{f}_i}/I_{{f}_i}$. Hence, Lemma \ref{lemma-classification-of-lagrangian-subgroups}
applies also to describe the set of Lagrangian subgroups of $I^*/I$.

Set $L:=\Lambda/I$ and  $L_i:=\phi_i^{-1}(\Lambda)/I$, $i=1,2$. Then $L\cap L_i=(0)$, by 
definition of $I_{\phi_i}$. Hence, $L_1=L_2$,
by Lemma \ref{lemma-classification-of-lagrangian-subgroups} (\ref{lemma-item-complementary-lagrangian-subgroups}). 
It follows that $\phi_1^{-1}(\Lambda)=\phi_2^{-1}(\Lambda)$. 
We conclude that $\phi_1\phi_2^{-1}$ belongs
to $O(\Lambda)$.
\end{proof}

\begin{lem}
\label{lemma-O-Lambda-orbits-of-I}
Let $\lambda_1$ and $\lambda_2$ be primitive elements of $\Lambda$ and $n>1$ an integer.
Set $I_i:=\left(\Lambda+{\rm span}_\ZZ\{\lambda_i/n\}\right)^*$, $i=1, 2$. There exists an integral isometry $g\in O(\Lambda)$, such that
$g(I_1)=I_2$, if and only if there exists an integer $k$, such that $\gcd(k,n)=1$, and
\begin{equation}
\label{eq-congruence-condition}
\frac{(\lambda_1,\lambda_1)}{2} \equiv k^2\frac{(\lambda_2,\lambda_2)}{2} \ \ \ \mbox{modulo} \ n.
\end{equation}
\end{lem}

\begin{proof}
Suppose first that such an isometry $g$ exists. Then $g(I_1^*)=I_2^*$. Hence,
$g(\lambda_1/n)+\Lambda=k\lambda_2/n+\Lambda$, for some integer $k$ relatively prime to $n$.
Hence, $g(\lambda_1)=k\lambda_2+n\alpha$, for some $\alpha$ in $\Lambda$.
Equality (\ref{eq-congruence-condition}) follows.

Assume next that Equation (\ref{eq-congruence-condition}) holds for an integer $k$
relatively prime to $n$. It suffices to find $g\in O(\Lambda)$ and $\alpha\in \Lambda$ such that 
$g(\lambda_1)=k\lambda_2+n\alpha$.
Set $d_i:=(\lambda_i,\lambda_i)/2$, $i=1, 2$. 
There exist integers $x$ and $y$ satisfying the equation
\[
\frac{d_1-k^2d_2}{n}=kx+ny,
\]
since $\gcd(k,n)=1$.
Let $T$ be the rank $2$ lattice with basis $\{t_1,t_2\}$ and Gram matrix
\[
\left(\begin{array}{cc}
(t_1,t_1) & (t_1,t_2)
\\
(t_2,t_1) & (t_2,t_2)
\end{array}
\right) \ \ = \ \ 
\left(\begin{array}{cc}
2d_2 & x
\\
x & 2y
\end{array}
\right).
\]
There exists a primitive isometric embedding $\tau \colon T\hookrightarrow \Lambda$, by a result due to Nikulin, 
see \cite [Lemma 8.1] {Lemma 81}.
There exists also and isometry $h$, such that $h(\tau(t_1))=\lambda_2$, since $\tau(t_1)$ and $\lambda_2$
are both primitive elements of $\Lambda$ of the same degree $2d_2$.
Hence, we may assume that $\tau(t_1)=\lambda_2$. 
Set $\alpha:=\tau(t_2)$. Then $\beta:=k\lambda_2+n\alpha$ is a primitive element of 
$\Lambda$ satisfying $(\beta,\beta)=2d_1$. Hence, there exists an integral isometry $g$ satisfying $g(\lambda_1)=\beta$.
\end{proof}

\begin{proof} [Proof of Proposition \ref{prop-only-one-double-orbit}.]
We may assume, without loss of generality, that $\phi_i=\widetilde{f}_i$, where 
$f_i$ belongs to $O(U_\QQ)$, by Lemma \ref{lemma-epsilon-is-surjective}. 
We may further assume that $f_1(e_1)=\frac{a}{b}e_1$, 
$f_1(e_2)=\frac{b}{a}e_2$, $f_2(e_1)=\frac{c}{d}e_1$, 
$f_2(e_2)=\frac{d}{c}e_2$, for positive integers $a$, $b$, $c$, $d$ satisfying $ab=n=cd$ and $\gcd(a,b)=1=\gcd(c,d)$. 
Note that $I_{f_1}^*/U$ is generated by $\frac{be_1+ae_2}{n}+U$ 
and $I_{f_2}^*/U$ is generated by $\frac{de_1+ce_2}{n}+U$.с
Set $\lambda_1:=be_1+ae_2$ and $\lambda_2:=de_1+ce_2$. 
Then $(\lambda_1,\lambda_1)=2n=(\lambda_2,\lambda_2)$. 
Hence, there exists an isometry $g\in O(\Lambda)$, such that
$g(I_{\tilde{f}_1})=I_{\tilde{f}_2}$, by Lemma \ref{lemma-O-Lambda-orbits-of-I}.

Note the equality 
${\displaystyle
g(I_{\tilde{f}_1})=g(\Lambda\cap \tilde{f}_1^{-1}(\Lambda))=\Lambda\cap g\widetilde{f}_1^{-1}(\Lambda)=I_{\tilde{f}_1g^{-1}}.
}$
The equality $I_{\tilde{f}_1g^{-1}}=I_{\tilde{f}_2}$ follows. Hence, $[\widetilde{f}_1g^{-1}]=[\widetilde{f}_2]$, by 
Lemma \ref{lemma-intersection-sublattice-determine-the-double-orbit}.
We conclude 
the desired equality of the double orbits $[\widetilde{f}_1]$
and $[\widetilde{f}_2]$.
\end{proof}

\subsection {An important example}
\label{An-important-example}

Let $S$ be a projective $K3$ surface, $v =(n,\alpha,s)\in H^0(S,\mathbb{Z}) \oplus H^{1,1}(S,\mathbb{Z})\oplus H^4(S,\mathbb{Z})$ be an isotropic primitive 
vector, $h$ be a $v$-generic ample divisor on $S$ and $M=M_h(v)$ be the smooth projective moduli space of  $h$-slope stable sheaves with Mukai vector $v$ on $S$.
Let $\mathcal E$ be a universal sheaf on $S \times M$.
The associated cohomology class
$$\kappa(\mathcal E^{\vee})\sqrt{td_{S \times M}}=\pi _S^*\sqrt{td_S}\cdot \kappa(\mathcal E^{\vee})\cdot \pi _M^*\sqrt{td_M}$$ in $H^*(S\times M, \mathbb{Q})$
determines the map $\varphi_{\mathcal E^{\vee}} \colon H^*(S, \mathbb{Q}) \rightarrow  H^*(M, \mathbb{Q})$, 
and thus determines the rational isometry
$$\psi_{\mathcal E} \colon H^2(S,\mathbb{Q}) {\rightarrow} H^2(M, \mathbb{Q})$$
(see Lemma \ref{Reformulation}).
We show here that $\psi_{\mathcal E}$ is of cyclic type.
Let $\Lambda$ be the cohomology lattice $H^2(S,\ZZ)$ and $I_{\psi_{\mathcal E}}= \psi^{-1}_{\mathcal E}(H^2(M, \mathbb{Z}))\cap H^2(S, \mathbb{Z})$.
We want to show that the quotient $\Lambda/I_{\psi_{\mathcal E}}$ is a cyclic group. 
First of all, $$\sqrt{\pi^*_Std_{S}} \cdot \sqrt{\pi^*_Mtd_ M}=1+\pi^*_S{\bf 1}_S+\pi^*_M{\bf 1}_M+\pi^*_S{\bf 1}_S\cdot\pi^*_M{\bf 1}_M \in H^*(S\times M,\ZZ),$$
where ${\bf 1}_S, {\bf 1}_M$ are fundamental classes of $S$ and $M$.
It is easy to see that the map $\psi_{\mathcal E} \colon H^2(S,\QQ) \rightarrow H^2(M,\QQ)$ 
is actually induced by $\kappa_2(\mathcal E^{\vee})$ only, so that we may
consider this class instead of $\kappa(\mathcal E^\vee)\sqrt{td_{S\times M}}$. 
The class $\kappa_2 (\mathcal E^{\vee})$ is easily expressed in terms of
$ch_2(\mathcal E^{\vee})$ and $c_1(\mathcal E^*)$ so 
that it becomes easy to find the domain in $H^2(S, \mathbb{Z})$ where
the map ${\psi}_{\mathcal E}$  induced by $\kappa_2(\mathcal E^{\vee})$ takes integral values.
Indeed, we see that as $$\kappa_2(\mathcal E^{\vee})=ch_2(\mathcal E^{\vee})-c_1^2(\mathcal E^{\vee})/2n$$ and 
$ch_2(\mathcal E^{\vee})$ is known to be integral, the question reduces to finding the domain where $c_1^2(\mathcal E^{\vee})/2n$
takes integral values. 
 The first Chern class of the restriction $c_1(\mathcal E^{\vee}|_{S\times\{m\}})$ is the 2-component $\alpha$
of the vector $v$, denote the other first Chern class restriction 
$c_1(\mathcal E^{\vee}|_{\{s\}\times M})$
by $\beta$. 
Then $$c_1^2(\mathcal E^{\vee})/2n=(\pi_S^*\alpha+\pi_M^*\beta)^2/2n=\pi_S^*\alpha^2/2n+\pi_S^*\alpha\cdot \pi_M^*\beta/n+\pi_M^*\beta^2/2n.$$

As the vector $v$ is isotropic, $\alpha^2=2ns$, so the fact that $\pi_S^*\alpha^2/2n$ is an integral class follows. 
It is also easy to see
that $\pi_M^*\beta^2/2n$ 
induces the zero map on $\Lambda$.
On the other hand, the mixed term $\pi_S^*\alpha\cdot \pi_M^*\beta/n$ induces on $\Lambda$
the map $$\gamma \mapsto (\gamma, \alpha)\beta/n.$$ Let us write $\alpha=k\cdot x, k\in \mathbb{Z}, \, x$ a primitive vector in
$\Lambda=H^2(S, \mathbb{Z})$ and respectively $\beta = j\cdot y, j \in \mathbb {Z}, \, y$ a primitive vector in $H^2(M, \mathbb{Z})$.
The subset of $\Lambda$, where ${\psi}_{\mathcal E}$ takes integral values, is
the sublattice $$I_{{\psi}_{\mathcal E}}=\left \{\gamma \in \Lambda {\rm \,such \,that \,}(\gamma,\alpha)j\, {\rm is\, divisible\, by\,} n\right \}=$$
$$=\left \{\gamma \in \Lambda {\rm \,such \,that \,}(\gamma,x)\, {\rm is\, divisible\, by\, }\frac{n}{{\rm gcd}(jk, n)}\right \} \subset \Lambda.$$


The latter description of $I_{\psi_{\mathcal E}}$ tells us that $\Lambda/I_{\psi_{\mathcal E}}$
is cyclic of order $\frac{n}{{\rm gcd}(jk, n)}$.
Indeed, $I_{\psi_{\mathcal E}}$ is the kernel of the group homomorphism 
$\Lambda \rightarrow \ZZ/m\ZZ$, where $m=\frac{n}{{\rm gcd}(jk, n)},$ given by $\gamma \mapsto (\gamma,x)$.
It is surjective due to unimodularity of $\Lambda$, so the 
above statement about the structure and the order of $\Lambda/I_{\psi_{\mathcal E}}$ follows.

\begin{Concl}
\label {Cyclicity-of-sheaf-induced-isometry}
 The rational Hodge isometry $\psi_{\mathcal E}$
is of $\frac {n}{{\rm gcd}(jk,n)}$-cyclic type.
\end{Concl}

\begin{Concl}
\label {Existence-of-n-type-isometry-for-any-n}
For any isometry $\phi:\Lambda_\QQ \rightarrow \Lambda_\QQ$
of $n$-cyclic type  there exist marked projective $K3$
surfaces $(S,\eta_S)$ and $(M,\eta_M)$,  where $M$ is a moduli space of rank $n$ sheaves on $S$, 
slope-stable with respect to an ample divisor $h \in Pic(S)$, a 
locally free family $\mathcal E$ over $S \times M$
and a rational Hodge isometry $\psi \colon H^2(S,\QQ) \rightarrow H^2(M,\QQ)$
induced by $\kappa(\mathcal E)\sqrt{td_{S\times M}}$,
such that $\phi=\eta^{-1}_S \psi \eta_M$.
\end{Concl}

\begin{proof}
Let us first construct $S, M$ and $\mathcal E$ and then pick up the markings.
Choose an integer $s>0$ such that gcd($n,s$)=1. Set $d=sn$.
Choose now a K3 surface $S$ 
so that $S$ is a general member of the family 
of (quasi-) polarized $K3$-surfaces of degree $2sn$  introduced in Mukai \cite{Mukai2}. Then, in particular, the Picard group of $S$ is cyclic, $Pic (S)=\ZZ h$.

Consider  the Mukai vector $v=(n, h,s)  \in \widetilde{H}(S, \mathbb{Z})$ and note that $ (v,v)=0$. 
Take  $M=M_{h}(v)$ to be the moduli space of rank $n$ locally free sheaves on $S$ that are $h$-Gieseker-stable.
As $Pic(S)$ is cyclic and for every $\mathcal F_m, m \in M,$ we have   $c_1(\mathcal F)=h$, 
the notions of $h$-slope-stability and $h$-Gieseker-stability on the surface $S$ are equivalent.
So 
$M$ is a moduli space of $h$-slope-stable sheaves on $S$.
%

The manifold $M$ is a $K3$ surface rational Hodge isometric to $S$, with Picard group $\ZZ \hat{h}$ for a primitive ample divisor $\hat{h} \in Pic(M)$, $\hat{h}^2=h^2$, see \cite[Prop. 1.1]{Mukai2}. 
There exists a universal sheaf $\mathcal E_0$ on $S \times M$, which follows from our choice of $v=(n,h,s)$
and  \cite[Thm. A.6, Rem. A.7]{Mukai1}.
One can construct from $\mathcal E_0$ a {\it normalized universal}
sheaf ${\mathcal E}$ on $S \times M$ such that $c_1(\mathcal E) = \pi^*_Sh+k\pi^*_M\hat{h}$
for an appropriate $k$, $sk \equiv 1$ (mod $n$). Mukai proved that the sheaf $\mathcal E$ is locally free.
For the proofs and constructions we refer to his paper \cite[Thm. 1.1, Thm. 1.2]{Mukai2}. 
Now we refer to
Conclusion \ref {Cyclicity-of-sheaf-induced-isometry}, setting $j=1$ and $n$ and $k$ being just our $n$ and $k$.
According to the Conclusion
$\kappa(\mathcal E^*)\sqrt{td_{S\times M}}$ induces a Hodge isometry of $n$-cyclic type,
which is the isometry $\psi$ in the formulation of the theorem.

 Let us show the existence of the required markings $\eta_S,\eta_M$. 
Take any markings $\eta_1, \eta_2$ of the surfaces $S$ and $M$.
Then the rational cyclic isometry $\phi_1=\eta_2 \psi \eta^{-1}_1 \colon  \Lambda_\QQ \rightarrow \Lambda_\QQ$
is also of $n$-cyclic type and thus, by Proposition \ref{prop-only-one-double-orbit}, determines the same double orbit $[\phi_1]$
as $\phi$ does, $[\phi_1]=[\phi]$. From here we see that we can pick up $\eta_S, \eta_M$
that take $\psi$ to $\phi$ as required above.
\end{proof}

\section {Moduli spaces of marked Hodge isometric K3s}

\label{Moduli-spaces}
In this section we realize steps 2) and 3) in our strategy sketch in the introduction.
In Subsection \ref{The-twisted-period-domain} we introduce the notion of a twisted period domain,
which is a generalization of the period domain notion with respect to the locus mentioned
in step 1).
In Subsection \ref{Definition-of-the-moduli-spaces} we introduce  two moduli spaces.
One moduli space parametrizes, roughly speaking, pairs of $K3$ surfaces $S_1,S_2$ together with a Hodge isometry $\varphi:H^2(S_1,\QQ) \rightarrow H^2(S_2, \QQ)$. This is our locus in step 1).
The other moduli space parametrizes, roughly, pairs 
of $K3$ surfaces $S_1,S_2$ together with a holomorphic vector bundle $\mathcal E$
on their product $S_1\times S_2$. This construction is closely related to
the previous one and is used in explaining  step 4) in the next two sections.

\subsection {The twisted period domain}
\label{The-twisted-period-domain}
\subsubsection{Marked K3 surfaces and their moduli space}
\label{Marked-K3-surfaces-and-their-moduli-space}
 The $K3$ lattice is the unique up to isomorphism even unimodular lattice of signature (3,19).
Fix such a lattice $\Lambda$
and denote by $(\cdot,\cdot)$ the corresponding bilinear form. 
Introduce the notation
$\Lambda_{\mathbb{F}}=\Lambda  {\otimes_\ZZ}\mathbb{F}$ for $\mathbb{F}=\QQ,\RR$ or $\CC$.
The vector space $\Lambda_\mathbb{F}$ is equipped with a bilinear pairing determined by $(\cdot,\cdot)$.
We will denote this pairing by $(\cdot,\cdot)$ as well.
By $[x]$ we denote the class of a vector $x\in \Lambda_\CC$
in the projective space $\mathbb{P}\Lambda_\CC$.
Let $$\Omega_\Lambda=\{[x]\in \mathbb{P}\Lambda_\CC {\rm\,\, such\,\, that\,\,} (x,x)=0 {\rm \,\,and\,\,} (x,\bar{x})>0\}$$  
be the period domain of $K3$ surfaces, constructed from $\Lambda$.
Sometimes, instead of $[x]$ we will be using the shorter notation $l$  for a period in $\Omega_\Lambda$.
Below by $\sigma_S$ we denote a global holomorphic form on a $K3$ surface $S$,
the letter $\eta$ with or without subscript stands for a marking of a $K3$ surface, that is, an isometry 
$\eta: H^2(S,\ZZ) \rightarrow \Lambda$.
We will use the same notation $\eta$ for both
the isometry $\eta \colon  H^2(S,\ZZ) \rightarrow \Lambda$ and the isomorphism of vector spaces
$H^2(S,\mathbb{F}) \rightarrow \Lambda_{\mathbb{F}}$ induced by $\eta$.
Denote by $\mathfrak{M}$ the moduli space of marked $K3$ surfaces and let
$$\Pe \colon  \mathfrak{M} \rightarrow \Omega_\Lambda,$$
$$(S, \eta) \mapsto [\eta(\sigma_S)],$$
be the period map.

Let us discuss now the construction of the locus in $\mathfrak M \times \mathfrak M$ 
from step 2) in the introduction.
The first step is to construct the corresponding period domain which would contain
the image of our locus under the Cartesian product $(\Pe,\Pe): \mathfrak M \times \mathfrak M \rightarrow
\Omega_\Lambda \times \Omega_\Lambda$ of the period map $\Pe$. 
We do not need to deal specifically with isometries of cyclic type now, so we consider the most general
case, specializing it step by step as needed.
We need to make certain formal conventions about markings.

\begin{rem}
\label {Marking-remark}
For a general complex manifold $X$ we say that a marking is an isomorphism of abelian groups $\eta_X:H^*(S, \ZZ) \rightarrow L$, 
where $L$ is some abelian group.
For a manifold $X=S_1\times S_2$, where $S_1,S_2$ are $K3$ surfaces, 
we want to consider only markings $\eta_X$ determined by (isometric) markings $\eta_1, \eta_2$ of $S_1$ and $S_2$ via K\"unneth decomposition of
$H^*(S_1\times S_2, \ZZ)$.  The most interesting part for us, of $H^*(S_1\times S_2, \ZZ)$, will be concentrated in
degree 4, namely, in $H^2(S_1,\ZZ)\otimes H^2(S_2,\ZZ)\subset H^4(S_1 \times S_2,\ZZ)$.
So, by a marking of $S_1\times S_2$ we mean
the isomorphism 
$\eta_1\otimes \eta_2 : H^2(S_1,\QQ)\otimes H^2(S_2,\QQ)$. 
We will denote it,  by $\eta_{S_1\times S_2}$ or by $\eta_{(I_1,I_2)}$
for complex structures $I_1,I_2$ on $S_1,S_2$ respectively, 
when $S_1\times S_2$ is considered as a member of a 
(topologically trivial) family of products of $K3$'s. 

In particular, given any linear mapping $\psi : H^2(S_1,\QQ) \rightarrow H^2(S_2,\QQ)$,
considering it as a class in $H^2(S_1,\QQ)^*\otimes H^2(S_2,\QQ) \cong H^2(S_1,\QQ)\otimes H^2(S_2,\QQ) \subset 
H^{4}(S_1\times S_2,\QQ)$,
we notice that $\phi\colon = \eta_1\otimes \eta_2 (\psi)=\eta_2\circ \psi \circ \eta_1^{-1}:\Lambda_\QQ \rightarrow \Lambda_\QQ$.
\end{rem}

Let $\pi \colon  \mathcal X \rightarrow T$ be a smooth and proper holomorphic map of complex manifolds with
a continuous trivialization $\eta \colon  R^k\pi_*\ZZ \rightarrow (\Lambda)_T$ of the local system, where 
$(\Lambda)_T$ is the trivial local system with fiber $\Lambda$. We get an induced flat trivialization of
the local system $R^k\pi_*\CC$, which we denote by $\eta$ as well. Let 0 be a point of $T$.
Denote by $X_t$ the fiber of $\pi$ over $t \in T$. The flat deformation of a class $\alpha \in H^k(X_0,\CC)$
in the local system $R^k\pi_*\CC$ associated to the family is given by the section 
$t \mapsto \eta^{-1}_t\eta_0(\alpha)$ of the local system $R^k\pi_*\CC$.

This together with Remark \ref{Marking-remark} shows that, in the above notations, given a family $\pi \colon  \mathcal X \rightarrow T$ of marked products of $K3$'s,
$((X_t\times Y_t), (\eta_{X_t}\otimes \eta_{Y_t}))$,
the family of rational Hodge isometries $$\psi_t=\eta_{X_t \times Y_t}^{-1}\eta_{X_0 \times Y_0}(\psi_0)
=\eta^{-1}_{Y_t}\eta_{Y_0}\psi_0\eta^{-1}_{X_0}\eta_{X_t}=\eta^{-1}_{Y_t}\phi\eta_{X_t}, t\in T,$$
determines a flat deformation of the Hodge isometry $\psi_0 \colon H^2(X_0,\QQ) \rightarrow H^2(Y_0,\QQ)$,
this deformation is considered as a (flat) section of $R^4\pi_*\QQ$.

Now we see that the locus from step 2) is actually a locus along which
the flat deformation of the class $\psi \in H^4(S_1 \times S_2,\QQ)$ stays of Hodge type.

Let us fix a rational isometry $\phi$ of $\Lambda_\QQ$. Then $\phi$ induces in a natural way 
an isomorphism $\widetilde {\phi} \colon \Omega_\Lambda \rightarrow \Omega_\Lambda$.
Consider the pairs $(S_1\times S_2, \eta)$ of
products of $K3$ surfaces $S_1, S_2$ together with marking
$$\eta_{S_1 \times S_2}=\eta_1\otimes \eta_2, \eta_i \colon  H^2(S_i, \ZZ)  \rightarrow \Lambda_\ZZ,$$ $ i=1,2, $
satisfying the condition that $$\psi=\eta^{-1}_2\phi \eta_1 \colon  H^2(S_1, \QQ) \rightarrow H^2(S_2, \QQ)$$ is a Hodge isometry.  Here the condition of being a homomorphism of Hodge structures certainly tells us
that $\widetilde{\phi}([\eta_1(\sigma_{S_1})])=[\eta_2(\psi(\sigma_{S_2}))]$, that is, 
pairs of periods $([\eta_1(\sigma_{S_1})],[\eta_2(\sigma_{S_2})])$ belong
to the graph of $\widetilde{\phi}$.
Denote the graph of $\widetilde{\phi}$ by $\Omega_{\phi}$. 
The set $\Omega_\phi$ serves as a period domain for the considered pairs of $K3$'s.
It is called {\it the twisted period domain associated to $\phi$}.

\subsubsection{Hyperk\"ahler manifolds}
\label {Hyperkahler-manifolds}
Let us recall the notion of a twistor line in $\Omega_\Lambda$. 
Fix a $K3$ surface $S$ and a marking $\eta \colon  H^2(S, \ZZ) \rightarrow \Lambda$. 
Take any positive 3-dimensional subspace $V$ in $H^2(S, \RR)$ containing the plane $P=\langle Re\,\sigma_S, Im\,\sigma_S \rangle$, and consider the image $\mathbb{P}\eta(V_\CC)$ of $ V_\CC=V\otimes_\RR \CC$ in $\mathbb{P}\Lambda_\CC$
under the natural projection.  
The intersection $Q_V$ of  $\mathbb{P}\eta(V_\CC)$ with $\Omega_\Lambda \subset \mathbb{P}\Lambda_\CC$
is a smooth complete conic in the plane $\mathbb{P}\eta(V_\CC)$, which is called {\it a twistor line through the point} $[\eta(\sigma_{S})]$
{\it corresponding to the subspace $V$}. 
For a detailed account of theory of twistor lines we refer to \cite{Huy2}.
Note that $V$ decomposes as an orthogonal sum $\RR \alpha \oplus \langle Re\, \sigma_{S}, Im\, \sigma_{S}\rangle$ for some $\alpha \in H^{1,1}(S, \RR)$ such that $(\alpha,\alpha)>0$. We will record this information denoting $Q_V$ by $Q_{S,\alpha}$.
There is an important case when $\alpha$ belongs to the K\"ahler cone of $S$.  In this case it is possible
to associate to the twistor line  a hyperk\"ahler structure on $S$. Let us describe this structure.
Recall (see \cite[p. 548]{Hitchin}) that a manifold $M$ is called {\it hyperk\"ahler} with 
respect to a metric $g$
if there exist covariantly constant complex structures $I,J$ and $K$ which
satisfy the quaternionic relations $$I^2=J^2=K^2=-1, IJ=-JI=K.$$ 
We call the  ordered triple $I,J,K$ {\it a hyperk\"ahler
structure on $M$ compatible with $g$}. 
A hyperk\"ahler structure $I,J,K$ gives rise to a sphere $S^2$ of complex structures
$$S^2=\{aI+bJ+cK| a^2+b^2+c^2=1\}.$$

Let us consider the case when our hyperk\"ahler manifold is a $K3$ surface $S$.
For any K\"ahler class $\alpha$ on $S$ with a complex structure $I$ there exists
a unique hyperk\"ahler metric $g$ such that $\alpha$ is a class in $H_{DR}^2(S, \mathbb{R})$ represented 
by the closed (1,1)-form $g(\cdot, I \cdot)$,
see \cite{Bea} and \cite[Ch. VIII]{Barth}. This hyperk\"ahler metric gives rise
to a hyperk\"ahler structure $I,J,K$ where $I$ is our original complex structure on $S$.
The choice of $J,K$ is not unique.
The sphere of complex structures $S^2$ can be identified with $\mathbb{P}^1$ as explained in \cite[p. 554]{Hitchin}.
Regarding  the identification $S^2 \cong \mathbb{P}^1$, the family $\{(S,\lambda)\}_{\lambda \in S^2}$ over the base $S^2 \cong \PP^1$ with the 
almost-complex structure induced in the horizontal direction by that of $\mathbb{P}^1$ and in the vertical ``fiber'' direction by $\lambda$, is a complex-analytic family (see \cite{Hitchin})
and so, by the definition of the moduli space of marked $K3$ surfaces $\mathfrak{M}$ it determines a (holomorphic) classifying map $S^2 \cong \mathbb{P}^1 \rightarrow \mathfrak{M}$. This map sends $\lambda \in S^2$ 
to the marked $K3$ surface $(S_\lambda,\eta_\lambda)$, here $\eta_\lambda \colon H^2(S_\lambda,\ZZ) 
\rightarrow \Lambda$ is determined by $\eta_S$ via topological isomorphism $H^2(S_\lambda,\ZZ) \cong H^2(S,\ZZ)$.
The image of $\mathbb{P}^1$ in $\mathfrak{M}$ under the classifying map
is the lift $\PP_{S, \alpha}$ of the conic $Q_{S, \alpha}$ with respect to $\Pe \colon \mathfrak{M} \rightarrow \Omega_\Lambda$ through the point $(S,\eta) \in \mathfrak{M}$.  We call such $\PP_{S, \alpha}\subset \mathfrak M$ (respectively $Q_{S, \alpha}\subset \Omega_\Lambda$) 
{\it a twistor line associated to a hyperk\"ahler structure} or {\it a hyperk\"ahler line}.
Introduce $V_\lambda:=\eta_\lambda^{-1}\eta_S(V) \subset H^2(S_\lambda,\RR)$.

Any complex structure $\lambda = aI+bJ+cK \in S^2$ determines a K\"ahler class $\omega_\lambda \in H^{1,1}(S_\lambda,\RR)$
represented by a closed positive (1,1)-form $g(\cdot, \lambda \cdot)$,
\begin{equation}
\label{omega-lambda}
\omega_\lambda=a\omega_I+b\omega_J+c\omega_K.
\end{equation}
The class $\omega_\lambda$ is orthogonal to the plane $P_\lambda=\langle Re\,\sigma_\lambda, Im\,\sigma_\lambda\rangle$
for a nonzero holomorphic 2-form $\sigma_\lambda \in H^{2,0}(S_\lambda)$.
The complexification $P_{\lambda,\CC}$ is spanned by $\sigma_\lambda$
and its conjugate $\overline{\sigma_\lambda}=\sigma_{-\lambda}$ and of the two forms
determined by $P_\lambda$ the form $\sigma_\lambda$, and thus the complex structure $\lambda$, is uniquely determined
by the fact that $\omega_\lambda$ is K\"ahler.
So every twistor family  over a hyperk\"ahler line $Q_{S,\alpha}$
is a family of $K3$ surfaces with a prescribed 
K\"ahler
class.

The class $\omega_J+i\omega_J$ up to proportionality over $\CC^*$ is the class of $\sigma_S\in H^{2,0}(S)$.
This tells us that 
we have the decomposition
\begin{equation}
\label{omega-lambda-decomposition}
V= \RR\alpha \oplus \langle Re\,\sigma_S, Im\,\sigma_S\rangle = \RR\omega_\lambda \oplus P_\lambda.
\end{equation}

\begin{rem}
\label{orthonormality-remark}
For the classes $\omega_I,\omega_J,\omega_K$ from the orthogonality relation
$\omega_I \perp \omega_J+i\omega_K$ we obviously have that $\omega_I\cdot \omega_J=\omega_I\cdot \omega_K=0$.
The fact that $\omega_J+i\omega_K$ is a class of a holomorphic 2-form means that
it is isotropic, $(\omega_J+i\omega_K)\cdot (\omega_J+i\omega_K)=0$, which implies that $\omega_J\cdot \omega_J=\omega_K\cdot \omega_K$
and $\omega_J\cdot \omega_K=0$. 
One can actually see that all $\omega_I,\omega_J, \omega_K$ have equal length: $ \omega_I \cdot \omega_I=\omega_J \cdot \omega_J
=\omega_K \cdot \omega_K$. This can be seen either from
the dicussion before Proposition 13.3 in \cite[Ch. VIII]{Barth} or
directly from the fact that if we start with the K\"ahler class $\omega_J$ on $(S,J)$
then the complex structures $K$ and $I$ can be obtained in the way described above,
so the cyclic permutation of our triple of complex structures indeed shows that $\omega_K\cdot \omega_K=\omega_J\cdot \omega_J$.
This shows that $\omega_I,\omega_J,\omega_K$ is an orthogonal basis of $V$ with
basis vectors having equal length. Moreover, it is clear that for all $\lambda=aI+bJ+cK\in S^2$ 
the classes 
%
$\omega_\lambda=a\omega_I+b\omega_J+c\omega_K $
 all have equal length.
\end{rem}
From now on we assume that the metric $g$ is normalized so that $\omega_\lambda$
is of length 1.

\begin{rem}
\label{verb-remark}
Note that a twistor line in $\Omega_\Lambda$ through a point $[\sigma_S]$, for $S$ such that $Pic(S)$ is trivial, is always
a line associated to a hyperk\"ahler structure on $S$. This is due to the fact that for such surfaces the K\"ahler cone is equal to
the positive cone. For this see Huybrechts, \cite{Erratum} or \cite [Prop. 5.4]{Bourbaki}, or  Verbitsky, \cite [Sec. 6]{Verbmapping}.
\end{rem}

\subsubsection{Generalities on hyperk\"ahler structures on products of K3's}
\label{Generalities-on-hyperkahler-structures}
Choose a positive 3-subspace $V\subset H^2(S_1, \RR)$ containing $\langle Re\, \sigma_{S_1}, Im \,\sigma_{S_1}\rangle$.
The projectivization $\mathbb{P}\eta_1(V_\mathbb{C})$ is a 2-plane in $\mathbb{P}\Lambda_\mathbb{C}$
and its intersection with the open subset of the quadric $Q$ determined by $(\cdot, \cdot)$ 
is a conic $Q_V$, namely, the twistor line through the point $[\eta_1(\sigma_{S_1})]$. 
\begin{defi}
The {\it twistor line $Q_{\phi, V}$ in $\Omega_\phi$ through $([\eta_1(\sigma_{S_1})], [\eta_2(\sigma_{S_2}))])$} is the graph $\Gamma_{\widetilde{\phi}}$ of the map $$\widetilde{\phi} \colon  Q_V \rightarrow \Omega_\Lambda,$$

$$[x] \mapsto [\phi(x)].$$
\end{defi}
The fact that the period of $(S_1\times S_2, (\eta_1,\eta_2))$ belongs to $\Omega_\phi$ means that 
$\psi=\eta^{-1}_2\phi \eta_1$ is a Hodge isometry and
$\widetilde{\phi}([(\eta_1(\sigma_{S_1}))]) = [\eta_2(\psi(\sigma_{S_2}))]$.
\begin{defi}
\label{Twistor-path-in-Omega-phi}
A {\it twistor path} in $\Omega_\phi$ consists of a finite ordered sequence $Q_1,Q_2, \dots, Q_k$ of twistor lines in $\Omega_\phi$,
such that every two consecutive lines intersect.
\end{defi}
Note that two twistor lines either intersect at precisely one point or do not intersect at all,
this is actually true for twistor lines in $\mathfrak M$ already.

Consider the twistor line $Q_{S_1,h}$ in $\Omega_\Lambda$ through $[\eta_1(\sigma_{S_1})]$ which  is determined by a K\"ahler class 
$h\in H^{1,1}(S_1, \RR)$. 
As earlier we consider the Hodge isometry $\psi = \eta_2^{-1}\phi \eta_1$.
Assume that $\psi(h)$ is a K\"ahler class in $H^2(S_2,\RR)$. 
Then the corresponding twistor line $Q_{\phi, V} \subset \Omega_\phi$
will be denoted  by $Q_{\psi, h}$. 
 The classes $h$, $\psi(h)$ and the original complex structures $I_1$ on $S_1$ and $I_2$ on $S_2$
determine hyperk\"ahler metrics $g_1$ and $g_2$ on $S_1$ and $S_2$.
Choose hyperk\"ahler structures $J^{'}_l,K^{'}_l$ compatible with metrics $g_l$, $l=1,2$.
We get a hyperk\"ahler structure $I:=I_1 \oplus I_2$, $J^{'}:=J^{'}_1 \oplus J^{'}_2, K^{'}:=K^{'}_1 \oplus K^{'}_2$
on $S_1 \times S_2$ compatible with the metric $g:=g_1 \oplus g_2$.
Again, like above the choice of $I,J^{'}, K^{'}$ is not unique.
Let us denote a particular choice of $I,J,K$ on $S_1 \times S_2$ by $\tau$.
Denote by $\pi \colon  \mathcal Y_\tau \rightarrow \PP^1_\tau$
the twistor family corresponding to the hyperk\"ahler structure $\tau=(I,J,K)$ on $S_1 \times S_2$.
The markings $\eta_1, \eta_2$ determine the map $\Pe_\tau \colon  \PP^1_\tau \rightarrow \Omega_\Lambda \times \Omega_\Lambda$ obtained as the composition of the classifying map $\PP^1_\tau \rightarrow \mathfrak M \times \mathfrak M$,
sending the particular point $\pi(S_1 \times S_2, I_1\oplus I_2) \in \PP^1_\tau$
to $\bigl(((S_1,I_1), \eta_1), ((S_2,I_2), \eta_2)\bigr) \in \mathfrak M \times \mathfrak M$, with the Cartesian product of the period map $(\Pe,\Pe) \colon  \mathfrak M \times \mathfrak M \rightarrow  \Omega_\Lambda \times \Omega_\Lambda$.
The classifying map $\PP^1_\tau \rightarrow \mathfrak M \times \mathfrak M$ determines a 
deformation of the isometry $\psi$, considered as a family of marked pairs of $K3$'s, 
$((S_{1,\lambda_1}, \eta_{\lambda_1}), (S_{2,\lambda_2}, \eta_{\lambda_2}))$ where 
$\lambda=(\lambda_1,\lambda_2)=(aI_1+bJ_1+cK_1,aI_2+bJ_2+cK_2)\in S^2$, namely
$$\psi_\lambda=\eta^{-1}_{\lambda}\circ \eta_{(I_1,I_2)}(\psi)=\eta^{-1}_{\lambda_2}\eta_2\psi \eta_1^{-1}\eta_{\lambda_1}.$$ 

\begin{lem}
\label{lemma-hyperkahler-twistor}
In the above notations, for any hyperk\"ahler structure $(I_1,J_1,K_1)$ on $S_1$
there exists a unique choice of hyperk\"ahler structure $(I_2,J_2,K_2)$
on $S_2$, 
such that the map $\Pe_\tau$ for $\tau=(I_1\oplus I_2,J_1\oplus J_2, K_1 \oplus K_2)$
sends $\PP^1_\tau$
isomorphically onto $Q_{\psi,h} \subset \Omega_\phi$.
\end{lem}
\begin{proof}
In order to show the existence we need to find complex structures $J_2,K_2$
fitting into a hyperk\"ahler structure 
$(I_2,J_2,K_2)$ 
such that
for every $\lambda =(\lambda_1,\lambda_2)=(aI_1+bJ_1+cK_1,aI_2+bJ_2+cK_2)$ in the corresponding sphere $S^2$ we have
that 
$$\psi_\lambda \colon  H^2(S_{1,\lambda_1},\CC) \rightarrow H^2(S_{2,\lambda_2},\CC)$$ is a Hodge isometry. 
This is precisely the condition that 
$\bigl(\Pe_\tau((S_{1,\lambda_1},\eta_1)),\Pe_\tau((S_{2,\lambda_2},\eta_2))\bigr)$ belongs to $Q_{\psi,h}$.
%
For any fixed $\lambda$ the Hodge condition $\psi_\lambda(H^{2,0}(S_{1,\lambda_1}))=H^{2,0}(S_{2,\lambda_2})$
together with the 'graph' condition
$$\psi(\RR h \oplus\langle Re\, \sigma_{S_1}, 
Im\, \sigma_{S_1}\rangle)=\RR \psi(h) \oplus\langle Re\, \sigma_{S_2}, 
Im\, \sigma_{S_2}\rangle,$$ rewritten in terms of $\psi_\lambda$ as 
$$\psi_\lambda(\RR \omega_{\lambda_1} \oplus\langle Re\, \sigma_{S_{1,\lambda_1}}, 
Im\, \sigma_{S_{1,\lambda_1}}\rangle )=
\RR \omega_{\lambda_2} \oplus   \langle Re\, \sigma_{S_{2,\lambda_2}}, 
Im\, \sigma_{S_{2,\lambda_2}}\rangle$$
is equivalent 
to the condition 
$\psi_\lambda(\omega_{\lambda_1})=\pm \omega_{\lambda_2}$,
where $\omega_{\lambda_1}$ and $\omega_{\lambda_2}$ are (length 1) K\"ahler classes on $(S_1,\lambda_1)$
and $(S_2,\lambda_2)$ defined by Equation (\ref{omega-lambda}). In the view of the above identifications
and the fact that $V_{\lambda_1}, \psi_\lambda(V_{\lambda_1})$ do not vary,
this is equivalent to $\psi(\omega_{\lambda_1})=\pm \omega_{\lambda_2}$.
 Now, given
that $\omega_{\lambda_1} \in V, \omega_{\lambda_2} \in \psi(V)$ depend continuously on $\lambda$
and that $\psi(\omega_{I_1})=\omega_{I_2}$
we have that the condition that $\psi_\lambda$ is a Hodge isometry for all $\lambda$ 
is equivalent to the condition
\begin{equation}
\label{equation-equivariant-isometry}
\psi(\omega_{\lambda_1})= \omega_{\lambda_2},
\end{equation}
for all $\lambda=(\lambda_1,\lambda_2) \in S^2$.

As we have for granted that $\psi(\omega_{I_1})=\omega_{I_2}$, in order to satisfy Condition (\ref{equation-equivariant-isometry})
it is sufficient to find  $J_2$ and $K_2$ such
that $\psi(\omega_{J_1})=\omega_{J_2}$
and $\psi(\omega_{K_1})=\omega_{K_2}$. 
These two conditions determine $J_2$ and $K_2$ in a unique way, see the discussion before 
Remark \ref{orthonormality-remark}. Due to Remark \ref{orthonormality-remark} these complex structures
together with $I_2$ satisfy the quaternionic identities and thus indeed form a hyperk\"ahler structure.
\end{proof}

We are using now the above introduced notations
$\tau$, $\pi \colon  \mathcal Y_\tau \rightarrow \PP^1_\tau$
and the classifying map $\PP^1_\tau \rightarrow \mathfrak M \times \mathfrak M$,
sending the particular point $\pi(S_1 \times S_2, I_1\oplus I_2) \in \PP^1_\tau$ to $\bigl(((S_1,I_1), \eta_1), ((S_2,I_2), \eta_2)\bigr) \in \mathfrak M \times \mathfrak M$.
\begin{defi}
\label{hyper-product-definition}
The image of the base $\PP^1_\tau$
of the family $\pi \colon  \mathcal Y_\tau \rightarrow 
\PP^1_\tau$
in $\mathfrak M \times \mathfrak M$ under the classifying map
is called the {\it twistor line through $((S_1,\eta_1),(S_2,\eta_2)) \in \mathfrak M \times \mathfrak M$ determined by $\psi$ and $h$} and is denoted 
by
\begin{equation}
\PP_{\psi,h}.
\label{Twistor-line}
\end{equation} 
\end{defi}
The line $\PP_{\psi, h}\subset \mathfrak M \times \mathfrak M$ is a lift of $Q_{\psi, h}\subset \Omega_\phi\subset \Omega_\Lambda \times \Omega_\Lambda$ to $\mathfrak{M} \times \mathfrak{M}$
with respect to $(\Pe, \Pe) \colon  \mathfrak M \times \mathfrak M \rightarrow \Omega_\Lambda \times \Omega_\Lambda$,
so that $\PP_{\psi, h}\subset (\Pe, \Pe)^{-1}(\Omega_\phi)$. We met this kind of lines in $\mathfrak M \times \mathfrak M$ above, when defined the map $\mathcal P_t$. The reason to give them a name now is because
Lemma \ref{lemma-hyperkahler-twistor}
tells us that any line of the form $\PP_{\psi,h}$ is associated to
 a hyperk\"ahler structure on $S_1\times S_2$. 

Note that comparing to the case of twistor lines in $\mathfrak M$, which may or may not be associated to hyperk\"ahler
structures on $K3$ surfaces, we here define a twistor line on $\mathfrak M \times \mathfrak M$ specifically as one associated to a hyperk\"ahler structure (on a product of $K3$ surfaces).
For our purposes we need only this ``restricted'' definition. 

\begin{rem}
\label{Lifting-twisted-remark}
Again as in Remark \ref{verb-remark}, if $Pic(S_1)=Pic(S_2)=\langle 0\rangle$ for $K3$ surfaces $(S_1,I_1)$ and $(S_2,I_2)$
then $Pic (S_1\times S_2)=\langle 0\rangle$ and any twistor line $Q_{\phi, V}$ is associated to a hyperk\"ahler structure
on $S_1\times S_2$ determined by $(I_1,I_2)$ and a K\"ahler class $h+\psi(h) \in H^{1,1}(S_1\times S_2, \RR)$.
Thus there exists a lift $$\PP_{\phi, V}$$ of $Q_{\phi, V}$ to $\mathfrak M \times \mathfrak M$.
\end{rem}

\subsection {Definition of the moduli spaces}
\label{Definition-of-the-moduli-spaces}

Here we introduce two notions of a moduli space. 
Both of the moduli spaces we want to define involve the self-product of the moduli space of marked $K3$ surfaces.
As we have a decomposition into connected components, $\mathfrak M=\mathfrak M^{+}\cup\mathfrak M^{-}$,
the self-product inherits it in an obvious form 
$$\mathfrak M \times \mathfrak M=\mathfrak M^{++}\cup\mathfrak M^{+-}\cup\mathfrak M^{-+}\cup\mathfrak M^{--}.$$
The crucial part for the construction of moduli spaces is $\mathfrak M^{++}\cup\mathfrak M^{--}$.

For a $K3$ surfaces $S$ introduce the set 
$$C_S
= \{\alpha \in H^{1,1}(S, \RR) | (\alpha, \alpha)>0 \}.$$
This set has two connected components.
Denote by 

\begin{equation}
C^+_S
\end{equation}
 the connected component containing the K\"ahler cone of $S$.
By definition this is the {\it positive cone of the K3 surface S}.
Let us denote the K\"ahler cone of $S$ by $K_S$.
\begin{defi} A Hodge isometry $\varphi \colon  H^2(S_1, \mathbb{R}) \rightarrow H^2(S_2, \mathbb{R})$ is called
{\it signed} if it maps $C^+_{S_1}$ to $C^+_{S_2}$. It is called {\it non-signed} if it is not signed.
\end{defi}
Note that for any non-signed isometry $\varphi$ the isometry $-\varphi$ is signed.
Here we note that there is a naturally arising  orientation of the positive cone $C_{S_i}$ in $$H^2(S_i, \mathbb{R}) = 
((H^{2,0}(S_i, \CC)\oplus H^{0,2}(S_i, \CC))\cap H^2(S_i, \RR))\oplus H^{1,1}(S_i, \mathbb{R})$$ determined by $C^+_{S_i}$
in
 $H^{1,1}(S_i, \mathbb{R})$: any choice of $\alpha \in C^+_{S_i}$ gives a positive definite 3-space $V$
with an orientation determined by basis $(\alpha,Re \,  \sigma_{S_i}, Im \, \sigma_{S_i})$.
\begin{defi} An isometry $\phi \colon  \Lambda_\RR \rightarrow \Lambda_\RR$ is called {\it signed} if
for some (and, hence, for any) point $((S_1,\eta_1), (S_2,\eta_2)) \in 
(\Pe,\Pe)^{-1}(\Omega_\phi) \cap (\mathfrak M^{++}\cup \mathfrak M^{--})$ 
the (Hodge) isometry
$\eta^{-1}_2 \phi \eta_1 \colon  H^2(S_1,\RR) \rightarrow H^2(S_2,\RR)$ 
is a signed Hodge isometry.
\end{defi}

\subsubsection{The cohomological moduli space.}
Now we want to introduce the {\it cohomological moduli space}.
Fix 
a signed isometry $\phi \colon \Lambda_\QQ \rightarrow \Lambda_\QQ$.

\begin{defi} The {\it cohomological moduli space} associated to a signed isometry $\phi \colon \Lambda_\QQ \rightarrow \Lambda_\QQ$
is a topological space defined as the set ${\mathcal M}_{\phi}$ consisting of all quadruples $x=((S_1, \eta_1), (S_2, \eta_2))$ where  $(S_1, \eta_1), (S_2, \eta_2) $  are marked $K3$ surfaces and $\psi_x$ is a signed Hodge isometry, $$ \psi_x \colon  H^2(S_1, \mathbb{Q}) \rightarrow H^2(S_2, \mathbb{Q}),$$ which satisfies

$\bullet \,\, \psi_x=\eta_2^{-1} \phi \eta_1$;

$\bullet \,\, \psi_x (K_{S_1}) \cap K_{S_2} \neq \emptyset.$
\label{Definition-cohomological-moduli}
\end{defi}



The topological space $\mathcal M_\phi$ is the locus mentioned in the introduction.  

\begin{thm}
$\mathcal M_\phi$ is a complex-analytic non-Hausdorff manifold.
\end{thm}

\begin{proof}
To prove that $\mathcal M_\phi$ is a complex-analytic manifold
it is sufficient to show that $\mathcal M_\phi$ is an open subset of the complex-analytic manifold $(\Pe,\Pe)^{-1}(\Omega_\phi) \subset \mathfrak M \times \mathfrak M$ where $(\Pe,\Pe) \colon  \mathfrak M \times \mathfrak M \rightarrow \Omega_\Lambda \times \Omega_\Lambda$ is the Cartesian product
of the period map $\Pe$. Indeed, we have that $(\Pe,\Pe)^{-1}(\Omega_\phi)$ is a complex-analytic submanifold
in $\mathfrak M \times \mathfrak M$ since the period map $\Pe \colon \mathfrak M \rightarrow \Omega_\Lambda$
is a local analytic isomorphism and $\Omega_\phi$  is a complex-analytic submanifold in $\Omega_\Lambda \times \Omega_\Lambda$.

A point $((S_1, \eta_1), (S_2,\eta_2)) \in (\Pe,\Pe)^{-1}(\Omega_\phi)$ belongs to $\mathcal M_\phi$
precisely when $$\eta^{-1}_2 \phi \eta_1 (K_{S_1}) \cap K_{S_2} \neq \emptyset.$$
The fact that this  is an open condition is proved in \cite[Ch. VIII, Prop. 9.4]{Barth}.
\end{proof} 
\begin{rem} There is an obvious forgetful map $$\Pe_\phi \colon {\mathcal M}_{\phi} \rightarrow \Omega_{\phi},$$
$$x=((S_1, \eta_1), (S_2, \eta_2)) \mapsto ([\eta_1(\sigma_{S_1})],[\eta_2(\sigma_{S_2})]).$$
Here in the definition of $\Pe_\phi$ we certainly use that $\psi_x (H^{2,0}(S_1, \mathbb{C}))=H^{2,0}(S_2, \mathbb{C})$
and $\eta_2\psi_x=\phi \eta_1.$ 
Note that the fiber $\Pe_\phi^{-1}(x)$ is a $W(S_1)\times W(S_2)$-torsor which can be 
idenified with direct product of the sets of chambers of positive cones of $S_1$ and $S_2$.
Here $W(S_1), W(S_2)$ are the respective Weil groups.
\end{rem}

\vspace*{2mm}

\begin{defi}  
\label{Twistor-path-definition}
A {\it twistor path} in $\mathcal M_\phi$ consists of a finite ordered sequence $\PP_1,\PP_2, \dots, \PP_k$ of twistor lines in $\mathcal M_\phi$, 
such that every two consecutive lines intersect. 

A twistor path is {\it generic} if each of the chosen points of intersection of consecutive twistor lines
corresponds to a pair of $K3$ surfaces with trivial Picard groups.

\end{defi}
It was shown in \cite{Bea-Geometrie}, see the more recent exposition in \cite{Bourbaki}, that every two periods $l_1,l_2 \in \Omega_\Lambda$ can be joined by a 
generic twistor path in $\Omega_\Lambda$. Now taking the graph of this twistor path under $\widetilde{\phi}$ in $\Omega_\phi$
we get a generic twistor path joining periods $(l_1,\widetilde{\phi}(l_1))$ and $(l_2,\widetilde{\phi}(l_2))$
in $\Omega_\phi$. Indeed, this graph is a generic twistor path in $\mathcal M_\phi$ 
because 1) products of $K3$'s corresponding to points $(q_i, \widetilde{\phi}(q_i)), q_i \in Q_i \cap Q_{i+1},$ 
have trivial Picard groups (as  Hodge isometric $K3$ surfaces have equal Picard numbers), and 2), as $\phi$ is signed, each of the lines of the graph is associated
to a hyperk\"{a}hler structure by Remark \ref{Lifting-twisted-remark} so 
it determines a twistor line from Definition \ref{hyper-product-definition}.

The following lemma will be used in the proof of Proposition \ref{prop-connected-components}.

\begin{lem}
\label{Kahler-cones-intersection}
 Let $x=((X,\eta_X),(Y,\eta_Y))$ be a point in $\mathcal M_\phi$, $\alpha\in K_X, \psi_x(\alpha) \in K_Y$.
Then the twistor line $\PP_{\psi,\alpha}\subset \mathfrak M \times \mathfrak M$ lies in
$\mathcal M_\phi$.
\end{lem}
\begin{proof}

The cohomology class $\alpha$ determines the line $Q=Q_{\psi,\alpha} \subset \Omega_\phi$.
Consider a connected lift of $Q$ to $\mathfrak M \times \mathfrak M$ containing the point $x$ 
Depending on whether the point $x$ belongs to $\mathfrak M^{++}$ or $\mathfrak M^{--}$  we have that all of the
 points $x_\lambda=((X_\lambda, \eta_{X_\lambda}), (Y_\lambda, \eta_{Y_\lambda})) \in 
\PP_{\psi, \alpha}$ that are lifts of points $\lambda \in Q \cong S^2$, are contained in either
$\mathfrak M^{++}$ or $\mathfrak M^{--}$.

Next, there is a distinguished K\"ahler class $\omega_\lambda \in K_{X_\lambda}$ which was introduced earlier. 
In order to prove that $x_\lambda$ belongs to $\mathcal M_\phi$ for every $\lambda \in Q$ 
it is sufficient  to check that $\psi_\lambda(\omega_\lambda) \in K_{Y_\lambda}$ for all $\lambda \in S^2$.
As $\psi_\lambda$ is a signed Hodge isometry we have that 
$$(\psi_\lambda (\omega_\lambda), Re\, \psi_\lambda(\sigma_{X_\lambda}), Im\, \psi_\lambda (\sigma_{X_\lambda}))$$
of the corresponding positive vector space $$\psi_\lambda(\langle \omega_\lambda, Re\, \sigma_{X_\lambda}, Im\, \sigma_{X_\lambda}
\rangle)
=\langle \psi_\lambda(\omega_\lambda), Re\, \sigma_{Y_\lambda}, Im\, \sigma_{Y_\lambda}\rangle$$ is positively oriented.
%
%
%
%
By the same 'initial condition and continuity argument' as in Lemma
\ref{lemma-hyperkahler-twistor} the class $\psi_\lambda(\omega_\lambda)$ must be K\"ahler. 
Now we see that the whole twistor line  $\PP_{\psi, \alpha}$ lies in $\mathcal M_\phi$.
\end{proof}

In the proofs of Proposition \ref {prop-connected-components}
and Proposition \ref{Surjectivity-prop}
below we will need to know how to construct a generic twistor line 
in $\mathfrak M \times \mathfrak M$ passing through a given point $x \in \mathcal M_\phi $. 

\begin{lem}
\label{generifying}
Given $x=((X,\eta_X), (Y,\eta_Y)) \in \mathcal M_\phi$ we can find a generic twistor line $\PP_{\psi_x,h} \subset \mathfrak M \times \mathfrak M$ 
for some $h \in K_X \cap \psi_x^{-1}(K_Y)$.
Moreover,  in the case of $X$ with cyclic Picard group, $Pic(X) \cong \ZZ h, h \in  Pic(X) \cap K_X \cap \psi_x^{-1}(K_Y)$, the twistor line $\PP_{\psi_x, h}$ 
is generic.
\end{lem}

\begin{proof}
If $Pic(X \times Y)$ is trivial then the statement of the lemma is tautologically true, any $h \in K_X$ works.
Let us assume now that $Pic(X) \neq \{0\}$.
We are going to find $h \in V:=K_{X}\cap \psi_x^{-1}(K_{Y})$ such that the twistor line $Q_{\psi, h} \subset \Omega_\phi$ 
contains the period of some $X^{\prime} \times Y^{\prime}$ with $Pic(X^{\prime})$ trivial. Then the corresponding, via $\phi$, surface $Y^{\prime}$ will also have a trivial Picard group. So
we will get $Pic(X^{\prime} \times Y^{\prime})=0$.  
Thus, the problem is reduced to finding $h$ as above such that the twistor line 
$$Q_{X,h}=\mathbb{P}\eta_X(\langle h, Re \, \sigma_{X}, Im\, \sigma_{X} \rangle)_{\mathbb{C}}\cap Q \subset \Omega_\Lambda,$$
contains the period $(t,\widetilde{\phi}(t))$ of a $K3$ surface with trivial Picard group.

Denote the subspace $\langle h, Re \, \sigma_{X}, Im\, \sigma_{X}\rangle$ for $h \in \Lambda_\RR$ by $W_h$.
In order for $Q_{X, h}$ to contain a period corresponding to a surface with trivial
Picard group we need 
\begin{equation}
\label{trivial-Picard}
\eta_X(W_h) \not\subset \underset{0\neq \lambda \in \Lambda}{ \cup}\lambda^{\perp},
\end{equation}
here the orthogonal complements $\lambda^{\perp}$ are taken in $\Lambda_\RR$.
Notice that $\eta_X(W_h) \subset \underset{0\neq \lambda \in \Lambda}{ \cup}\lambda^{\perp}$
is equivalent to existence of nonzero $\lambda \in \Lambda$ such that $\eta_X(W_h) \subset \lambda_h^{\perp}$.

Assume that condition (\ref{trivial-Picard}) is broken for every $h \in V$.
Then for every $h \in V$ we have some $\lambda$ such that
$\eta_X(W_h) \subset \lambda^{\perp}$ and thus we may write
$$V\subset\underset{0\neq \lambda\in \Lambda}{ \cup}\lambda^{\perp},$$
So the open subset $V$ of $H^{1,1}(X,\RR)$ can be expressed as a countable union of its subsets
that are cut out by linear subspaces in $H^{1,1}(X,\RR)$
and thus there must be  $\lambda$ such that $V \subset \lambda^{\perp}$.
For this $\lambda$ we have 
$W_h\perp\eta^{-1}_X(\lambda)$ for every $h \in V$.
This means that $H^{2,0}(X, \RR) \oplus H^{0,2}(X, \RR) \perp \eta^{-1}_X(\lambda)$ and $H^{1,1}(X, \RR)\perp \eta^{-1}_X(\lambda)$
which contradicts the non-degeneracy of the form $(\cdot,\cdot)$.  
So there exists an $h$ such that $W_h$ determines a twistor line
containing a point $t$ with $X_t$ having trivial Picard group. Fix this $h$.

As $\psi$
is a rational isometry, the Picard group of $Y_{\widetilde{\phi}(t)}$ is also trivial.
Thus, we can find a period $(t,\widetilde{\phi}(t)) \in Q_{\psi_t, h} \subset \Omega_\phi$ determining surfaces $X^{\prime}, Y^{\prime}$ with trivial Picard groups
together with their markings and a rational isometry determining a point $((X^{'}, \eta_{X^{'}}),(Y^{'}, \eta_{Y^{'}})) \in \mathcal M_\phi $ with $Pic (X^{'}\times Y^{'})=\langle 0 \rangle.$ 

If $Pic(X) \cong \ZZ h$ for $h \in V$, then 
condition (\ref{trivial-Picard}) is satisfied. Indeed, if  we had
$\eta_X(W_h) \subset \lambda_h^{\perp}$ then like earlier $\eta^{-1}_X(\lambda_h)$ lies in $H^{1,1}(X,\ZZ)\cong \ZZ h$
and so is a (nonzero) integral multiple of $h$. This contradicts to the condition $h \perp \eta^{-1}_X(\lambda_h)$. 
Thus $\PP_{\psi-x,h}$ is generic.
\end{proof}

\begin{prop}
\label{prop-connected-components}
 The moduli space ${\mathcal M}_{\phi}$ consists of 2 connected components ${\mathcal M}^+_{\phi}
= {\mathcal M}_\phi \cap \mathfrak M^{++}$ and ${\mathcal M}^-_{\phi}={\mathcal M}_\phi \cap \mathfrak M^{--}$. 
Moreover, for any two points in the same connected component there exists a twistor path in the corresponding connected component consisting of lines of the form (\ref{Twistor-line})
and connecting these points.

\end{prop}

\begin{proof}

First of all, set-theoretically $\mathcal M_\phi = \mathcal M^+_\phi \sqcup \mathcal M^-_\phi$.
Now let us prove the connectedness of each of the two sets in the decomposition.
Let us prove that ${\mathcal M}^+_{\phi}$ is connected.

We are going to show that any two points of $\mathcal M^+_\phi$ can be joined
by a twistor path.
First, pick up a point $y= ((S, \eta_{S}),(T, \eta_{T}))\in {\mathcal M}_\phi$ 
with periods $l:=\eta_S(H^{2,0}(S,\CC))$ and 
$\widetilde{\phi}(l)=\eta_T(H^{2,0}(T,\CC))$ in $\Omega_\Lambda$, 
such that the Picard groups $Pic (S)$ and $Pic (T)$ are trivial. 
We can certainly find such $(S, \eta_S)$  and $(T, \eta_T)$ with periods
$l$ and $\widetilde{\phi}(l)$ respectively,
because of surjectivity of the period map $\Pe \colon  \mathfrak{M} \rightarrow \Omega_\Lambda$ and the fact that
periods corresponding to surfaces with nontrivial Picard groups are contained in a countable union of hyperplanes
$\underset{\lambda \in \Lambda}\cup \mathbb{P}\lambda^{\perp}_\CC \subset \mathbb{P}\Lambda_\CC$.
Once we choose $l$ so as to have $Pic(S)=\langle 0 \rangle$, the group $Pic (T)$ is also trivial because $\psi=\eta^{-1}_T\phi\eta_S$ is a rational Hodge isometry,
hence the Picard groups of $S$ and $T$ are isomorphic.

Moreover, $\psi_y=\eta^{-1}_T\phi\eta_S$ is a signed Hodge isometry, so it takes $C^+_S$
to $C^+_T$. 
Again, as for $K3$ surfaces with trivial Picard group
the K\"ahler cone equals the positive cone, 
we conclude that the point $y$ indeed
belongs to $\mathcal M^+_\phi$.
Such a choice of $y$ means that $(S, \eta_{S})$ and $(T, \eta_{T})$ are the only points in the fibers of
the period map $\pi \colon  \mathfrak{M} \rightarrow \Omega_\Lambda$ over $l$ and $\widetilde{\phi}(l)$.
Indeed, let, for example, $(X, \eta_{X})$, $(Y, \eta_Y)$ be two distinct points in ${\mathcal M}_{\phi}$.
such that $$\pi(X, \eta_{X})=\pi(Y, \eta_Y)=l.$$
Now $g=\eta^{-1}_{Y}\eta_{X} \colon H^2(X, \mathbb{Z}) \rightarrow H^2(Y, \mathbb{Z})$ is a signed Hodge isometry.
So, by the strong Torelli theorem for $K3$ surfaces, see \cite{Piat}, there is an isomorphism $f \colon  Y \rightarrow X $ and $w \in W(Y)$, the Weil group of $Y$, generated by
reflections with respect to hyperplanes in $H^2(Y, \mathbb{Q})$ orthogonal to divisors corresponding to $(-2)$-curves on $Y$,
such that $g=w\circ f^*$.
But $W(Y)$ is trivial by the choice of $Y$. So $g=f^*$ and we have that $\eta_Y\circ f^* =\eta_X$, so that
$f^*$ induces an isomorphism of marked pairs $(X, \eta_X),(Y, \eta_Y)$ which means that they represent the same
point of moduli space of marked $K3$ surfaces.  The same argument applies to the fiber of $\pi$
over $\widetilde{\phi}(l)$.
For such choice of $l$ and $\widetilde{\phi}(l)$ the point $y$ itself is a unique point in the fiber of $\pi_\phi$ over   $(l,\widetilde{\phi}(l))$.

Now choose an arbitrary point $$x=((X, \eta_{X}),(Y, \eta_{Y})) \in \mathcal M^+_\phi.$$ 
We want to connect this arbitrary point to the point $y$ by a twistor path in $\mathcal M^+_\phi$ consisting of lines of the form
(\ref{Twistor-line}) which is going to be constructed as a connected lift of a generic twistor path joining the 
periods $\Pe_\phi(x)$ and $\Pe_\phi(y)$ in $\Omega_\phi$. This would prove the connectedness
of  $\mathcal M^+_\phi$.

Given a generic path of twistor lines $Q_1,\dots,Q_k$, $q_i \in Q_i \cap Q_{i+1}, i=1,\dots, k-1,$
in $\Omega_\phi$ joining $\Pe_\phi(x)$ and $\Pe_\phi(y)$,
there exist lifts of $Q_2, \dots Q_k$ to $\mathfrak M \times \mathfrak M$ via lines of the form (\ref{Twistor-line}), by 
 the discussion after Definition \ref{Twistor-path-definition} and by Remark \ref{Lifting-twisted-remark}.
 The union of their lifts is connected by the above discussed uniqueness of the preimages of $q_1, \dots, q_{k-1}$
under $\Pe_\phi$. We need only to make sure that a lift of $Q_1$ through $x$ exists.
Existence of $Q_1$ with this property 
is provided by Lemma \ref{generifying}. A point $z \in Q_1$ 
with trivial Picard group can be taken to be $q_1$, and then $Q_2,\dots,Q_k$ may be taken to be any generic twistor path joining $z$ to $y$. 
 Then the whole path of twistor lines admits a (connected) lift of the required form to $\mathfrak{M} \times \mathfrak{M}$ through the point $x=((X, \eta_{X}),(Y, \eta_{Y}))$.

Moreover, as
$x$ was chosen to be in $\mathfrak M^{++}$, the whole twistor path will be contained in
$\mathfrak M^{++}$ and thus in ${\mathcal M}_\phi$.
Now using our special choice of $y$ we see that we thus get a path consisting of lines
of the form $$\lambda \mapsto ((X_\lambda, \eta_{X_\lambda}),(Y_\lambda, \eta_{Y_\lambda}))$$ in the self-product $\mathfrak{M}\times \mathfrak{M}$ joining $x$ and $y$ and being a lift of the corresponding twistor path in $\Omega_\phi$ with respect to the obvious projection. 
Moreover, by construction $\psi_\lambda=\eta^{-1}_{Y_\lambda}\phi\eta_{X_\lambda}$
is a rational Hodge isometry. 
Each of the lines of this path lies in $\mathcal M^+_\phi$ by Lemma \ref{Kahler-cones-intersection}.
Now we see that the whole twistor path lifts to $\mathcal M^+_\phi$ and, moreover, it connects $x$ to $y$ because 
it contains a lift of the $\Pe_\phi(y)$ with respect to $\Pe_\phi$ and $y$ was chosen so that it is the only point in the fiber of $\Pe_\phi$
over its period.

Thus we know that we can join any point $x \in \mathcal M^+_\phi$ to $y\in \mathcal M^+_\phi$. 
Finally, we see that we can join any $x,z \in \mathcal M^+_\phi$ by lifts of twistor paths.

The proof of connectedness of ${\mathcal M}^-_{\phi}$ goes analogously.
This completes the proof of Proposition \ref{prop-connected-components}.
\end{proof}

\subsubsection{The sheaf moduli space.}

\label{Sheaf-moduli-space}
\begin{defi} The {\it sheaf moduli space} associated to a signed isometry $\phi \colon \Lambda_\QQ \rightarrow \Lambda_\QQ$
is 
the set $\widetilde{\mathcal M}_\phi$ of quintuples 
$((S_1, \eta_1), (S_2, \eta_2), {\mathcal E})$, where $(S_i, \eta_i), i=1,2,$ are marked  $K3$ surfaces and
${\mathcal E}$ is a twisted  sheaf on $S_1 \times S_2$ satisfying the following conditions 

$\bullet$ $\mathcal E$ is a locally free sheaf;

$\bullet$ the cohomology class $-\kappa(\mathcal E^\vee)\sqrt{td_{S_1\times S_2}}$ determines
a Hodge isometry \\\hspace*{0.6cm} $\psi_{\mathcal E} \colon H^2(S_1,\QQ) \rightarrow H^2(S_2,\QQ),$ such that 
$\psi_{\mathcal E}(K_{S_1})\cap K_{S_2} \neq \emptyset$;

$\bullet$ $\psi_{\mathcal E}=\eta^{-1}_2\phi \eta_1$.
\end{defi}

Note that here we define $\psi_{\mathcal E}$ as the isometry determined by  $-\kappa(\mathcal E^\vee)\sqrt{td_{S_1\times S_2}}$
instead of $\kappa(\mathcal E^\vee)\sqrt{td_{S_1\times S_2}}$ as in Section \ref {A new formulation of Mukai's result}.
The reason for introducing the minus sign is that having in mind our important example of Subsection \ref{An-important-example} we want 
our Hodge isometry to satisfy the second requirement of definition of $\widetilde{\mathcal M}_\phi$,
that is, the K\"ahler cone condition. 
That $-\kappa(\mathcal E^\vee)\sqrt{td_{S_1\times S_2}}$ satisfies this condition will essentially be shown in the proof of Proposition \ref{not-empty}. 

The topological space $\widetilde{\mathcal M}_{\phi}$ has a structure of a  complex-analytic non-Hausdorff manifold.
For our purposes here it is sufficient to consider the sheaf moduli space
set-theoretically.

\begin{nota} The forgetful map $$\widetilde{\Pe} \colon  \widetilde{\mathcal M}_{\phi} \rightarrow {\mathcal M}_{\phi},$$
is the map defined by

$$\widetilde{x}=((S_1, \eta_1), (S_2, \eta_2), {\mathcal E})  \mapsto x=((S_1, \eta_1), (S_2, \eta_2)),$$
and, certainly, $\psi_x=\psi_{\mathcal E}$.

\end{nota}

\begin{rem} 
%
It is clear
that for points $m$ in the image of $\widetilde{\Pe}$ the corresponding $\psi$ is represented by a class of analytic type
(which was defined in the introduction)
in $H^*(X \times Y, \mathbb{Q})$. Now let $X$ and  $Y$ be algebraic $K3$ surfaces. Then $\psi$ is algebraic. So the fact that rational Hodge isometries between algebraic $K3$ surfaces, appearing as a coordinate
in such a quintuple in $\mathcal M_\phi$,  are algebraic would follow if the map $\widetilde{\Pe}$ were surjective.
\end{rem}

 Introduce the following subsets of $\widetilde{\mathcal M}_{\phi}$,
$$\widetilde{\mathcal M}^+_{\phi}=\widetilde{\Pe}^{-1}({\mathcal M}^+_{\phi}),$$
and
$$\widetilde{\mathcal M}^-_{\phi}=\widetilde{\Pe}^{-1}(\mathcal M^-_{\phi}).$$

For examples illustrating the introduced notions we refer
to Subsection \ref{An-important-example}.
According to the general example there it appears  that the rational Hodge  isometry determined by $\kappa(\mathcal E^\vee)\sqrt{td_{S \times M}}$
for $\mathcal E$ a universal sheaf  on $S\times M$ 
are of a very specific type. Via markings of the corresponding surfaces such isometries induce rational isometries 
$\phi \colon  \Lambda_\QQ \rightarrow \Lambda_\QQ$ of {\it cyclic type}, a notion that was introduced in the Introduction. Recall the equivalent
definition of isometries of cyclic type given in Section \ref{Double-orbits}.
We say that a rational isometry $\phi \colon  \Lambda_\QQ \rightarrow \Lambda_\QQ$ is of {\it $n$-cyclic type}
if $\Lambda/(\Lambda \cap \phi^{-1}(\Lambda)) \cong \ZZ/n\ZZ$. 
Isometries of cyclic type are, in some sense, the simplest isometries of $\Lambda_\QQ$.
Moreover, for every such isometry $\phi$ we can find algebraic $K3$ surfaces $S_1, S_2$ and appropriate markings so that $\phi$
determines a rational Hodge isometry of the cohomology lattices of the surfaces and this Hodge isometry is
induced by the $\kappa$-class of a twisted sheaf.
This is proved in the following proposition.

\vspace*{2mm}
\begin{prop} 
\label{not-empty} 

For a signed isometry $\phi$ of cyclic type the moduli space $ \widetilde{\mathcal M}_{\phi}$ is not empty.
\end{prop}

\begin{proof}
Let $n$ be the natural number such that $\phi$ is of $n$-cyclic type.
The proof of the proposition is essentially given by
Conclusion \ref{Existence-of-n-type-isometry-for-any-n}, which provides
existence of  marked $K3$ surfaces $(S,\eta_S)$, $(M,\eta_M)$ and a 
locally free sheaf $\mathcal E$ over $S\times M$, such that
$\kappa(\mathcal E)\sqrt{td_{S \times M}}$ induces an $n$-cyclic type isometry $\psi \colon  H^2(S,\QQ) \rightarrow H^2(M,\QQ)$.
The only thing that needs to be verified is that $\psi_{\mathcal E}=-\psi$ satisfies the K\"ahler cone condition
in the definition of $\widetilde{\mathcal M}_\phi$.

As $\hat{h}^2=h^2$ and $\psi_{\mathcal E}$ restricts to an isomorphism between $\QQ h$ and $\QQ\hat{h}$,
we certainly have that $\psi_{\mathcal E}(h)=\pm \hat{h}$.
Let us show that $\psi_{\mathcal E}(h)=\hat{h}$.
The isometry $\psi_{\mathcal E}$ is induced by the $H^4(S\times M,\QQ)$-component of $-\kappa(\mathcal E^\vee)\sqrt{td_{S\times M}}$. We have 
$$\kappa(\mathcal E^\vee)\sqrt{td_{S \times M}}=
ch(\mathcal E^\vee)e^{-\frac{c_1(\mathcal E^\vee)}{n}}\sqrt{td_{S \times M}}=$$ 
$$=\left (n+c_1(\mathcal E^\vee)+\frac{c_1(\mathcal E^\vee)^2-2c_2(\mathcal E^\vee)}{2}+\dots \right) \cdot \left (1-\frac{c_1(\mathcal E^\vee)}{n}+\frac{c_1(\mathcal E^\vee )^2}{2n^2}-\dots \right)\times$$ 
$$\times(1+\pi^*_S{\bf 1}_S+\pi^*_M{\bf 1}_M+\pi^*_S{\bf 1}_S\cdot\pi^*_M{\bf 1}_M)=$$
$$=n+\left (\frac{(n-1)c_1(\mathcal E^\vee)^2}{2n}-c_2(\mathcal E^\vee)+\pi^*_S{\bf 1}_S+\pi^*_M{\bf 1}_M\right)+\dots .$$
As $\pi^*_S{\bf 1}_S$ and $\pi^*_M{\bf 1}_M$ induce zero maps $H^2(S,\ZZ) \rightarrow H^2(M,\ZZ)$ we need to consider only the term
$$\frac{(n-1)c_1(\mathcal E^\vee)^2}{2n}-c_2(\mathcal E^\vee).$$
In $$c_1(\mathcal E^\vee)^2=(-\pi^*_Sh-j\pi^*_M\hat{h})^2=\pi^*_S(2ns {\bf 1}_S)+\pi^*_M(2nsj^2 {\bf 1}_M)+2\pi^*_Sh\cdot \pi^*_Mj\hat{h}$$
the only component inducing a nonzero map is $2\pi^*_Sh\cdot \pi^*_Mj\hat{h}$.
Denote the image of $h$ in $H^2(M,\ZZ)$ under $c_2(\mathcal E^\vee)=c_2(\mathcal E)$ by $\psi$, following the notations of 
\cite{Mukai2}.  As Mukai showed in \cite[Prop. 1.1]{Mukai2},  $\psi= [1+(2(n-1)sj)]\hat{h}$. The image of $h$ under $2\pi^*_Sh\cdot \pi^*_Mj\hat{h}$ is $4nsj\hat{h}$.
Finally we get that $\psi_{\mathcal E}$ maps $h$ to 
$${\pi_M}_*\left (\pi^*_Sh\cdot \left (c_2(\mathcal E^\vee)-\frac{(n-1)c_1(\mathcal E^\vee)^2}{2n}\right)\right )=\psi-\frac{4(n-1)nsj\hat{h}}{2n}=$$ $$=\psi-2(n-1)sj\hat{h}=\hat{h}.$$
So the K\"ahler cone condition for $\psi_{\mathcal E}$ is verified and the proof of the proposition is complete.
\end{proof}

\section{Hyperholomorphic sheaves}
\label{Hyperholomorphic-sheaves}
In this section we formulate some results on extending sheaves over
twistor families. These results will be used in the next section.
The original definition of a hyperholomorphic vector bundle can be found in  \cite[Def. 3.11]{Verb-book}.
Here we give a more convenient for us definition which is equivalent to the original one.

Assume $N$ is a hyperk\"ahler manifold, $I$ is a fixed complex structure on $N$ and $h$ is a K\"ahler class on $N$. There exists a hyperk\"ahler metric
associated to a (1,1)-form representing the cohomology class  $h$. 
Then we have a sphere of complex structures on $N$, a twistor family $\mathcal N \rightarrow S^2$, and a K\"ahler
class $\omega_\lambda$ on the fiber $N_\lambda$ for each $\lambda \in S^2$, such that $N_I=N$ and $\omega_I=h$. 
\begin{rem}
\label {A-convenient-remark}
By definition, the class $\omega_\lambda$ (considered up to multiplying by a positive scalar) is {\it the K\"ahler class} on $(N,\lambda)$.
\end{rem}
 Fix these $h, S^2$ and $\mathcal N$.
A vector bundle $F$ on $(N,I)$ is called {\it h-hyperholomorphic}
if it can be extended to a vector bundle over $\mathcal N$.
The K\"ahler class $h$ defines the notion of $h$-slope-stability and $h$-slope-polystability of vector bundles on $N$. Recall that a vector bundle on $N$ is called $h$-{\it slope-polystable} if it is
isomorphic to a direct sum of $h$-slope-stable bundles with equal slopes, see an equivalent definition in
\cite[Def. 1.5.4 ]{Huybrechts-Lehn}. 

Let us formulate 
a theorem proved by M. Verbitsky (see \cite[Thm. 3.17, Thm. 3.19]{Verb-book}, keeping in mind
Remark \ref{A-convenient-remark}).

\begin{thm} {\it Let $F$ be an $h$-slope-polystable vector bundle over $(N,I)$.
If the Hodge types of $c_1(F)$ and $c_2(F)$ are preserved under the deformation
identified with the sphere $S^2$ of complex structures on $N$, then the bundle $F$ 
extends over $\mathcal N$.
Furthermore, the extended vector bundle $\mathcal F$ over $\mathcal N$ restricts to the fiber $N_\lambda$ as an 
$\omega_\lambda$-slope-polystable bundle, for each $\lambda \in S^2$.}
\label {Hyper-theorem}
\end{thm}

\begin{rem}
\label{deformation-invariance}
Let $\pi \colon  \mathcal X \rightarrow T$ be a smooth and proper holomorphic map of complex manifolds with
a continuous trivialization $\eta \colon  R^k\pi_*\ZZ \rightarrow (\Lambda)_T$ of the local system, where 
$(\Lambda)_T$ is the trivial local system with fiber $\Lambda$.
Let $\mathcal F$ be a vector bundle over $\mathcal X$.
Each Chern class of $\mathcal F$ defines a flat section of the local system.
Deformations of the Chern classes of the vector bundle  $F=\mathcal F_0$ on $X_0$ agree with taking Chern classes of the extension, that is, $\eta^{-1}_t\eta_0(c_i(\mathcal F_0)) =c_i(\mathcal F_t) \in H^*(X_t,\CC)$.
\end{rem}
%
%
We will be particularly interested in the case when $\mathcal X$ is a family of products of
marked $K3$ surfaces, $T=Q_{\psi, h}\subset \Omega_\phi$.

\begin{rem}
\label{remark-hyperkahler-twistor}
 Theorem \ref{Hyper-theorem} addresses existence of  extensions of 
sheaves on $N$ over a family over the sphere $S^2$ of complex structures on $N$. 
In Section \ref{Proof-of-the-main-result} we will need to extend sheaves over families of products of $K3$'s over twistor lines in $\Omega_\phi$. 
In order to use Theorem  \ref{Hyper-theorem} we need to associate to a twistor line $Q_{\psi,h}$ an appropriate sphere of complex structures $S^2$ on $N$.
While such association trivially exists in the case when $N$ is a $K3$ surface, we need to be careful
when $N$ is a product of $K3$'s. Lemma \ref{lemma-hyperkahler-twistor} provides a way to associate a hyperk\"ahler structure to a twistor line.
\end{rem}

We are now going to make some preparations needed for the proof of Proposition \ref{Surjectivity-of-p}
 and related to the problem of extending sheaves over twistor families. 
Let $\mathcal N,\mathcal F$ and $h$ be as in Theorem \ref{Hyper-theorem} for the rest of this section.
Consider now a hyperk\"ahler manifold $N=S\times M$  
where $(S,\eta_S)$ and $(M,\eta_M), M=M_{h_S}(v)$, $h_S$ is an ample divisor  on $S$,  
are marked projective $K3$ surfaces and $\mathcal E$ is a universal untwisted sheaf on $S \times M$.
We will specify the precise choice of $S,M$ and $\mathcal E$ that is used in the proof of Proposition \ref{Surjectivity-of-p}
in the formulation of Lemma \ref{stability} below.
Ideally we would like to be able to extend $\mathcal E$ over a twistor family through $((S, \eta_S), (M, \eta_M), \psi_{\mathcal E})$ determined by $H:=(h_S,\psi_{\mathcal E}(h_S))$ which corresponds to a
twistor line in $\mathcal M_{\phi}$ where 
$$\phi=\eta_M \psi_{\mathcal E}\eta^{-1}_S  \colon  \Lambda_\QQ \rightarrow \Lambda_\QQ,$$
$$ \phi \in \Lambda^*_\QQ \otimes \Lambda_\QQ.$$  
The Hodge types of $c_1(\mathcal E)$ and $c_2(\mathcal E)$ need not be preserved under such twistor deformation and thus we cannot, in general, hope for extending
$\mathcal E$ as an untwisted sheaf. 
%
The preservation of the Hodge type of $\psi_{\mathcal E}$ 
along $\mathcal M_\phi$ 
gives the first step towards understanding what should be our extendable substitute for $\mathcal E$.
Consider a point $((S, \eta_S), (M, \eta_M), \psi_{\mathcal E}) \in \mathcal M_\phi$
of the kind constructed in Proposition \ref{not-empty}.
The isometry $\psi_\mathcal{E} \colon H^2(S, \QQ) \rightarrow H^2(M,\QQ)$ is actually induced by the degree 4 component of the class $\kappa(\mathcal E^\vee)\sqrt{td_{S\times M}}$ as we saw in Subsection \ref{An-important-example}. 
Moreover, as the Hodge types of Todd classes of $S$ and $M$
are trivially preserved under any deformations,
the preservation of the Hodge type of $\psi_{\mathcal E}$ 
implies preservation of the Hodge type of $\kappa_2(\mathcal E^\vee)$.
And now we describe the last step towards finding an appropriate substitute for $\mathcal E$.
Consider the sheaf $\mathcal A = \mathcal End(\mathcal E) = \mathcal E^\vee \otimes \mathcal E$ 
instead of $\mathcal E$. We have $$c_1(\mathcal A)=0$$ and $$c_2(\mathcal A) = -2r \kappa (\mathcal E),$$ 
where $r= rk(\mathcal E)$.
Let us explain the last equality.
 Recall that $$\kappa(\mathcal A)=ch(\mathcal A)exp(-c_1(\mathcal A)/r^2)=ch(\mathcal A),$$ see Subsection \ref{A new formulation of Mukai's result}.
Now, comparing the components of $\kappa(\mathcal A)$ and $ch(\mathcal A)$ of degree 4 we get $$c_2(\mathcal A)=-ch_2(\mathcal A)=-\kappa_2(\mathcal A)
=-\kappa_2(\mathcal E^\vee \otimes  \mathcal E)=-2r\kappa_2(\mathcal E).$$ 
In the last equality
we used the multiplicativity of the class $\kappa$ on tensor products, the fact that $\kappa_1$ of any sheaf
is zero and $\kappa_{2i} (\mathcal E) = \kappa_{2i}(\mathcal E^\vee)$ for any $i \geqslant 0$.
So the class $-2r\kappa(\mathcal E)$ naturally serves as the second Chern class for the sheaf of  Azumaya algebras $\mathcal A$. The Hodge type of the class $\kappa(\mathcal E)$ is preserved under the considered twistor deformations together with the Hodge type of $\kappa(\mathcal E^\vee)$,
and thus  the Hodge type of $c_2(\mathcal A)$ is preserved under the considered twistor deformations.
Thus the condition of Hodge type invariance in Theorem \ref{Hyper-theorem} for 
$c_1(\mathcal A)=0$ and $c_2(\mathcal A)$ is satisfied.

The only condition that remains to check is that $\mathcal A$ is $H$-slope-polystable.
Note that in order to prove
that $\mathcal A$  is $H$-slope-polystable for which it suffices to show that $\mathcal E$ is.
Indeed, then by \cite[Thm. 3.2.11]{Huybrechts-Lehn} we would have that $\mathcal A=\mathcal E^\vee \otimes \mathcal E$
is $H$-slope-polystable and we may thus set $\mathcal F:=\mathcal A$.
Thus by Theorem \ref{Hyper-theorem} the sheaf $\mathcal A$ can be deformed as an untwisted sheaf. 
\begin{rem}
\label{preservation-of-algebra-structure}
More precisely,
this twistor deformation preserves the structure of Azumaya algebra, see Section 6 of \cite{M}. 
The explanation is based on the fact that for any $h$-hyperholomorphic $h$-polystable sheaf of zero slope, 
any global section is indeed flat with respect to the Hermite-Einstein
connection determined by $h$, and is thus holomorphic with respect to all complex structures $\lambda$
in the twistor family. 
Now, the multiplication $m \colon \mathcal End \,\mathcal E \otimes \mathcal End \,\mathcal E \rightarrow
\mathcal End \,\mathcal E$ determined by the composition is a global section of the sheaf 
$\mathcal Hom(\mathcal End \,\mathcal E \otimes \mathcal End \,\mathcal E, \mathcal End \,\mathcal E)$
which is $h$-polystable again by \cite[Thm. 3.2.11]{Huybrechts-Lehn}, 
$h$-hyperholomorphic, because so is $\mathcal End \,\mathcal E$, and of zero slope. Thus $m$ is holomorphic with respect to all 
twistor deformations determined by $h$.
 
\end{rem}
As we know from the correspondence between sheaves of Azumaya algebras and twisted sheaves (see, for example,  \cite{Cald}) deforming 
the sheaf  $\mathcal A$ as a sheaf of Azumaya algebras is the same
as deforming the sheaf $\mathcal E$ as a twisted sheaf. So Remark \ref{preservation-of-algebra-structure}
tells us that the twistor deformation of $\mathcal A$ provided by
Theorem \ref{Hyper-theorem} determines a deformation of $\mathcal E$ as a twisted sheaf.

Let us get to the problem of proving the stability of $\mathcal E$.
In general the universal sheaf $\mathcal E$ need not be slope-stable with respect to
an arbitrary divisor $H$ but there exists a particular choice of $S,M,\mathcal E$ and $H$
such that $\mathcal E$ is $H$-slope-stable.

\begin{lem}
\label {stability}
Consider $K3$ surfaces $S$, $M$, ample divisors $h_S \in Pic (S)$, $h_M \in Pic(M)$, and the universal sheaf $\mathcal E$ on $S\times M$ as in the proof of Proposition \ref{not-empty}.
Then   $\mathcal E$ is slope-stable with respect to $\pi^*_Sh_S+\pi^*_Mh_M \in Pic(S\times M)$.
Consequently, the Azumaya algebra $\mathcal A :=\mathcal End(\mathcal E)$ is $(\pi_S^*h_S+\pi_M^*h_M)$-hyperholomorphic.
\end{lem}
\begin{proof}

The slope-stability of $\mathcal E$ with respect to the divisor $\pi^*_Sh_S + \pi^*_Mh_M$ in  $Pic (S \times M)$ 
follows from the slope-stability of ${\mathcal E}|_{S \times \{m\}}, m \in M,$ with respect to $h_S \in Pic(S)$
(which follows from the definition of $M$ and $\mathcal E$), and the slope-stability of ${\mathcal E}|_{\{s\}\times M}, s\in S,$
with respect to $h_M \in Pic (M)$ (which follows from Mukai's paper \cite[Thm. 1.2]{Mukai2},
where in his  notation $h$ corresponds to our $h_S$ and $\hat{h}$ corresponds to our $h_M$). 
\end{proof}

\section{Proof of the main result}
\label{Proof-of-the-main-result}
The core result of this section is Proposition \ref{Surjectivity-prop}, which states that for each of the subsets  $\widetilde{\mathcal M}^\pm_{\phi} \subset \widetilde{\mathcal M}_{\phi}$, $\phi$ an isometry of cyclic type,
the restriction $\widetilde{\Pe} \colon  \widetilde{\mathcal M}^{\pm}_\phi \rightarrow \mathcal M^{\pm}_\phi$ sending $(S_1,S_2,\mathcal E)$ to $(S_1,S_2, -\kappa(\mathcal E^\vee)\sqrt{td_{S_1\times S_2}})$ is surjective. 
The proof of Proposition \ref{Surjectivity-prop} is given in Subsection \ref{Surjectivity-of-p}.
Subsection \ref{Reduction-to-the-case-of-a-cyclic-isometry} explains why Theorem \ref{Main-Theorem-Algebraic} formulated in the introduction follows from
Proposition \ref{Surjectivity-prop}. 
\subsection {Surjectivity of $\widetilde{\Pe}$}

\label{Surjectivity-of-p}
\vspace*{2mm}

In this subsection the isometry $\phi$ is assumed to be of cyclic type.
Let ${\mathcal M}_\phi^+$ and ${\mathcal M}_\phi^-$ be the connected components of ${\mathcal M}_\phi$ introduced in Proposition \ref{prop-connected-components} and $\widetilde{\mathcal M}^\pm_\phi$ the sets defined in 
Subsection \ref{Sheaf-moduli-space}.
The following proposition is proved
using the technique of hyperholomorphic sheaves developed by M. Verbitsky.

\begin{prop} 
\label{Surjectivity-prop}
The forgetful map $\widetilde{\Pe}  \colon  \widetilde{\mathcal M}_{\phi} \rightarrow {\mathcal M}_{\phi}$ 
restricts to a surjective map $\widetilde{\Pe}  \colon  \widetilde{\mathcal M}^{\pm}_{\phi} \rightarrow {\mathcal M}^\pm_{\phi}$.
\end{prop}

\begin{proof}
The arguments for each of the two choices of the sign $\pm$ are absolutely similar, so
we will be proving surjectivity of $\widetilde{\Pe}  \colon  \widetilde{\mathcal M}^{+}_{\phi} \rightarrow {\mathcal M}^+_{\phi}$. 
Let us first outline the proof of surjectivity.
For an arbitrary point $y =((S_1,\eta_1),(S_2,\eta_2), \psi_y) \in \mathcal M^+_{\phi}$ we want to find a preimage of this point in $\widetilde{\mathcal M}^+_{\phi}$
under $\widetilde{\Pe}$.
First, we pick up a point $\widetilde{x} \in \widetilde{\mathcal M}^+_{\phi}$,
which exists by Proposition \ref{not-empty}. 
Let $x=\widetilde{\Pe}(\widetilde{x}) \in \mathcal M^+_{\phi}$. 
We choose an appropriate connected twistor path $l \subset \mathcal M^+_{\phi}$
consisting of twistor lines $\PP_j, j=1, \dots,k,$ 
joining $x$ and $y$, $x\in \PP_1$, $y \in \PP_k$.  
We then show that there exist locally free twisted sheaves $\mathcal F_j$ over complex-analytic families 
$\mathcal X_j = (i|_{\PP_j})^*\mathcal Y$ of
products of marked $K3$'s, for the inclusion $i \colon  l\rightarrow \mathfrak M \times \mathfrak M$ and the universal family 
$\mathcal Y \rightarrow \mathfrak M \times \mathfrak M$ so that ${\mathcal X_k}|_y =(S_1 \times S_2,(\eta_1,\eta_2)) $. 
Thus $$y=\widetilde{\Pe}(((S_1,\eta_1),(S_2,\eta_2), \mathcal F_k|_{{\mathcal X_k}|_y}))$$
and we are done.

Let us now get to the details of the sketched construction.
We want and can choose the point $\overline{x}$ to be of the form $((S,\eta_S), (M, \eta_M), \mathcal E)$ with $M=M_{h_S}(v)$
and $\mathcal E$ a universal sheaf on $S \times M$.
Let us specify the choice of $S$ and $\mathcal E$.
Choose $S$ as in the proof of Conclusion \ref{Existence-of-n-type-isometry-for-any-n} 
 to be a $K3$ surface with a cyclic Picard group generated by an ample divisor $h_S$, so that $S$ is general in the sense of Mukai,   \cite{Mukai2}. 
Then $M=M_{h_S}(v)$ has also a cyclic
Picard group, as  rationally
Hodge isometric $K3$ surfaces have the same Picard numbers. 
Choose $\mathcal E$
to be a normalized universal sheaf on $S\times M$ as in the proof of Conclusion \ref{Existence-of-n-type-isometry-for-any-n}.
We choose $l$ to be  a generic twistor path in $\mathcal M_\phi$ joining $\widetilde{\Pe}(\widetilde{x})$ and $y$
and consisting of lifts $\PP_j$ of lines $Q_j \subset \Omega_\phi$ associated to hyperk\"ahler structures.
Such a path was shown to exist in the course of the proof of Proposition \ref{prop-connected-components},
where $\mathcal M^+_\phi$ was shown to be connected. 
Let $p_j$ be the pont in the intersection of lines $\PP_j$ and $\PP_{j+1}$.

We are going to extend consequently $\mathcal A=\mathcal End(\mathcal E)$ over the families $\mathcal X_j$. 
This means that having extended $\mathcal A$ over $\mathcal X_j$ we get the extension over $\mathcal X_{j+1}$ provided that
the conditions of Theorem \ref{Hyper-theorem} are satisfied for the result of extension of $\mathcal A$ restricted to 
$\mathcal X_j|_{p_j}$ for $p_j \in \PP_j \cap \PP_{j+1}$ for all $j=1,\dots,k-1$, $p_0:=x$.

In order to extend $\mathcal A$ from $\mathcal X_1|_{p_0}=((S,\eta_S),(M,\eta_M))$ over the whole $\mathcal X_1$, 
we want to choose $\PP_1$ in the path $l$ to be the line
determined by $h \in Pic(S)$. 
Indeed, as $\psi_{\mathcal E}(h_S)=h_M$, 
Lemma   \ref{generifying} guarantees that this twistor line is generic.
By the construction of $\mathcal M_{\phi}$ the Hodge type of $c_2(\mathcal A)$ under the twistor deformations defined by lines
in $\mathcal M_{\phi}$ (see the discussion after Remark \ref{Marking-remark})
and we proved $H$-slope-stability of $\mathcal E$ 
for $H=\pi_S^*h_S+\pi_M^*h_M$ in Lemma \ref{stability}.
So, by Remark \ref{remark-hyperkahler-twistor} and Theorem \ref{Hyper-theorem} we can extend $\mathcal A$ over the first twistor family over the first line of the path.
Further extensions over families $\mathcal X_2, \dots, \mathcal X_k$
exist automatically because  the sheaf $\mathcal A_t$ over $S_t \times M_t$ with trivial Picard group, $Pic(S_t)=Pic(M_t)=\langle 0 \rangle$,
being slope-polystable with respect to some K\"ahler class (by Theorem \ref{Hyper-theorem}),
is 
slope-polystable for any choice of K\"ahler class on $S_t \times M_t$. 
Indeed, on a compact K\"ahler manifold with a trivial Picard group  a slope-stable locally free sheaf of rank $n>0$
does not admit any subsheaf of rank $<n$, so it is automatically slope-stable with respect to
any K\"ahler class on this manifold. The same holds if we replace slope-stability by slope-polystability.



 The proof of Proposition \ref{Surjectivity-prop}
is complete.
\end{proof}

\subsection {Reduction to the case of a cyclic isometry.}
\label {Reduction-to-the-case-of-a-cyclic-isometry}

Here we prove Theorem \ref{Main-Theorem-Algebraic} 
formulated in the introduction. 

\begin{prop} Every rational Hodge isometry of  
projective surfaces of cyclic type is algebraic. 
\label{Cyclic-characteristic}
\end{prop}
\begin{proof}
Let $X,Y$ be any projective $K3$ surfaces and $f \colon  H^2(X,\QQ) \rightarrow H^2(Y,\QQ)$ be a rational Hodge isometry.
Up to replacing $f$ with $-f$ we may regard $f$ as a signed isometry.
We want to prove that $f$ is algebraic 
by applying Proposition \ref{Surjectivity-prop}.
In order to apply this proposition we need $f$ to satisfy the K\"ahler cone condition  from Definition \ref{Definition-cohomological-moduli},
namely $f(K_X)\cap K_Y \neq \emptyset$. 
As $f$ is signed, it identifies the positive cones of $X$ and $Y$, and  after 
composing with a finite number of reflections $r_C$ for (-2)-curves $C$, which are all
algebraic, one can assume that $f$ satisfies the K\"ahler condition.
Choose any signed markings
$\eta_X \colon H^2(X,\ZZ) \rightarrow \Lambda$ and
$\eta_Y \colon H^2(Y,\ZZ) \rightarrow \Lambda$ and set $\phi:=\eta_Y  f  \eta^{-1}_X$.
Thus the moduli space $\mathcal M_\phi$ is defined and we may now apply  Proposition \ref{Surjectivity-prop},
according to which the quadruple $((X,\eta_X),(Y,\eta_Y))$ lies in the image of the forgetful map $\widetilde{\Pe_\phi} \colon \widetilde {\mathcal M}_\phi \rightarrow \mathcal M_\phi$.
This proves that $f$ is algebraic. 
\end{proof}
Further by an isometry we mean an isometry of the rationalized $K3$ lattice $\Lambda_\QQ$.
The cyclic type isometries are not only simple looking ones, they are actually building blocks for all
other rational isometries, 
that is, every rational isometry $\phi$ of $\Lambda_\QQ$ is a composition of
isometries of cyclic type. This follows from Cartan-Dieudonn\'e
theorem,  see \cite[Prop. 2.36]{Gerstein}.
%
%
For the proof of Theorem \ref{Main-Theorem-Algebraic} we will need the following lemma.

\begin{lem}
\label {lemma-composition}

Let $\alpha \in H^*(S_1\times S_2,\QQ)$ and $\beta \in H^*(S_2\times S_3,\QQ)$ be algebraic classes for 
smooth projective varieties
$S_1,S_2,S_3$. Then the class ${\pi_{13}}_* (\pi^*_{12}(\alpha) \cup \pi^*_{23}(\beta)) \in H^*(S_1\times S_3,\QQ)$ is also algebraic.
\end{lem}
Note that the pull-back of an algebraic  class is an algebraic class and the product of two 
algebraic classes is also an algebraic class.
So the only operation in the lemma for which preservation of algebraicity is nontrivial
is the pushforward.
The preservation  of the algebraicity under a pushforward (and, hence, Lemma \ref{lemma-composition})
follows from the more general Lemma \ref{lemma-pushforward} formulated below.
In order to formulate this lemma we need to introduce a new notation.
Let $X$ be a smooth projective variety. Denote by $Hdg(X)_{alg}$
the subring of the cohomology ring $H^*(X,\QQ)$ generated by algebraic classes.

\begin{lem}
\label {lemma-pushforward}

Let $f \colon  X \rightarrow Y$ be a smooth morphism of smooth projective varieties $X$ and $Y$.
Then $f_*Hdg(X)_{alg} \subset Hdg(Y)_{alg}$.

\end{lem}

For the proof of  Lemma \ref{lemma-pushforward} we refer to Section \ref{Appendix}.

\begin{proof}[Proof of Theorem \ref{Main-Theorem-Algebraic}]

Consider a general Hodge isometry $\varphi  \colon  H^2(S, \QQ) \rightarrow H^2(S^{'}, \QQ)$ inducing via markings
$\eta, \eta^{'}$ an isometry $\phi  \colon  \Lambda_\QQ \rightarrow \Lambda_\QQ.$
The isometry $\phi$ decomposes as a composition of cyclic type isometries, $\phi =\phi_k \circ \dots \circ \phi_1$ by 
the Cartan-Dieudonn\'e theorem.
Set $\eta_0:=\eta$ and $\eta_k:=\eta^{\prime}$.
Starting with $S$ and using the surjectivity of the period map for marked $K3$ surfaces we obtain a sequence of marked $K3$ surfaces 
$$(S, \eta), (S_1,\eta_1) \dots, (S_{k-1}, \eta_{k-1}), (S^{'}, \eta^{'}),$$
with periods $[\eta(\sigma_S)], l_i=\widetilde{\phi}_i\circ \dots \circ \widetilde{\phi}_1([\eta(\sigma_S)]), i=1,\dots, k$,
so that $l_k=[\eta^{'}(\sigma_{S^{'}})]$. 

Set $\psi_i=\eta^{-1}_{i+1}\phi_{i+1}\eta_i$, $i=0,\dots, k-1$. 
Now, each $\psi_i$ is algebraic by Proposition \ref{Cyclic-characteristic} and is thus given by an algebraic class $$\alpha_i \in H^{2,2}(S_i \times S_{i+1}, \QQ).$$
Then taking the composition of the $\alpha_i$'s as correspondences by Lemma \ref{lemma-composition} we obtain an 
algebraic class  $$\alpha \in \oplus_p H^{p,p}(S_0 \times S_{k}, \QQ),$$
which induces the isometry $\varphi$. 
This completes the proof of Theorem \ref{Main-Theorem-Algebraic}.
\end{proof}

\section{Appendix}
\label{Appendix}

\begin{proof}[Proof of Lemma \ref{lemma-pushforward}]

  We will prove the lemma by proving that the vector space $Hdg(X)_{alg}$ is spanned over $\QQ$ by classes of the form $ch(\mathcal F)td_f$, for $\mathcal F$ a coherent sheaf on $X$ and $td_f$ the Todd class of the relative tangent bundle of the morphism $f \colon X \rightarrow Y$. Then the statement will follow by Grothendieck-Riemann-Roch formula 
$$f_*(ch(\mathcal F)\cdot td_f)=ch(\sum (-1)^iR^if_*(\mathcal F)),$$ 
and coherence of the sheaves $R^if_*\mathcal F$.

First of all note that the class $td_f$ belongs to $Hdg(X)_{alg}$ being an invertible element of this ring. Multiplication by 
$td_f$ induces a linear automorphism of the vector space  $Hdg(X)_{alg}$. So it is sufficient
to prove that  $Hdg(X)_{alg}$ is spanned  over $\QQ$ by classes of the form $ch(\mathcal F)$. 
In order to prove this we need two lemmas formulated below.
Let $V$ be the subring of $H^*(X,\QQ)$ generated by Chern characters $ch(\mathcal F)$
of coherent sheaves $\mathcal F$ on $X$.
\begin{lem}
\label {spanning-lemma}
The subring $V$ is 
spanned over $\QQ$ by classes of the form $ch(\mathcal F)$.
\end{lem}
\begin{lem}
\label{V-is-Hdg}
The subring $V$ is equal to $Hdg(X)_{alg}$.
\end{lem}
The idea of the proof of the first statement is based on multiplicativity
of Chern characters of locally free sheaves $\mathcal F, \mathcal G$ with respect to the tensor product, 
$$ch(\mathcal F)\cdot ch(\mathcal G) =  ch(\mathcal F \otimes \mathcal G).$$

The notion of the Chern character extends from locally free sheaves to coherent sheaves
via using locally free resolutions and the Whitney formula.
Let $K_{alg}(X)$ denote the K-group generated by coherent
sheaves on $X$. Every class in $K_{alg}(X)$ can be represented in the form $[A]-[B]$ for locally free sheaves $A$, $B$. 
This allows to extend the tensor product operation from classes of locally free sheaves
to classes of arbitrary coherent sheaves so that $ch \colon K_{alg}(X) \rightarrow H^*(X,\ZZ)$
is a ring homomorphism.

Note, that in the definitions of $Hdg(X)_{alg}$ and $V$ the Chern characters of coherent sheaves may be replaced by Chern characters of classes of $K_{alg}(X)$, which will not change our $V$ and $Hdg(X)_{alg}$
as subsets of $H^*(X)$. So this proves that the ring $V$ is spanned over $\QQ$ by Chern characters of coherent sheaves.
This completes the proof of Lemma \ref{spanning-lemma}.

Now let us prove Lemma \ref{V-is-Hdg}.
For this it is sufficient to show that any Chern class $c_i(\mathcal F)$ for $\mathcal F$ a coherent sheaf on $X$
belongs to $V$.  Actually, from the formula expressing the Chern character
of $\mathcal F$ as a polynomial of Chern classes of $\mathcal F$ it follows that it is sufficient to prove that $V$
contains together with every class $ch(\mathcal F)$ all its homogeneous components. Let us prove this statement first
for $\mathcal F$ being a vector bundle and then extend the idea of the proof to the general case.

Let  $\mathcal F$ be a locally free sheaf on $X$.
By definition $V$ contains Chern characters $ch(\mathcal F_k)$ of locally free sheaves
$\mathcal F_k$
defined by induction,  $$\mathcal F_0 =\mathcal F,$$ $$\mathcal F_k =\wedge^2\mathcal F_{k-1}, k \geqslant 1.$$
 The Chern characters of $\mathcal F_{k}$ and $\mathcal F_{k-1}$ are related in the following way. 

\begin{equation}
\label{Chern-character-of-wedge-product}
ch(\mathcal F_{k}) = \frac{(ch(\mathcal F_{k-1}))^2}{2}-r_2\,ch(\mathcal F_{k-1}),
\end{equation}
where $r_2$ is the automorphism of the cohomology ring $H^*(X,\QQ)$ acting by multiplication by $2^i$ on $H^{2i}(X,\QQ)$.
This equality follows from the splitting principle.

Equality (\ref{Chern-character-of-wedge-product}) tells us, in particular, that the subring of the ring $V$, generated by the Chern characters of 
locally free sheaves $\mathcal F$, is invariant under the automorphism $r_2$. 
So, together with every $ch(\mathcal F)$, the ring  $V$ contains all the elements 
of the form $$r^k_2 \, ch(\mathcal F), 0 \leqslant k \leqslant n={\rm dim}_\CC \,X.$$ 
Now a simple argument involving 
 nonvanishing of the corresponding Vandermonde determinant tells us that $V$ contains all the homogeneous components of $ch(\mathcal F)$.

As it was mentioned earlier, an arbitrary class in $K_{alg}(X)$ can be represented in the form $[A]-[B]$ for locally free sheaves $A$, $B$. 
Then we have
the wedge self-product defined as in \cite[Chapter III, \S 1]{AtiyahK}, that is
$$\wedge^2 ([A]-[B]) =[\wedge^2 A] - [A\otimes B] + [Sym^2 \, B].$$
Taking the Chern character of both sides of this equality we get
$$ch(\wedge^2 ([A]-[B])) = ch(\wedge^2 A)-ch(A\otimes B)+ch(Sym^2 \, B)=$$ $$=\frac{(ch(A))^2}{2}-\frac{r_2ch(A)}{2}-ch(A)ch(B)+ch(Sym^2 \, B).$$
Now the proof of (\ref{Chern-character-of-wedge-product}) in the general case reduces to the following lemma,
whose proof follows from the splitting principle as above.
\begin{lem} 
\label{ch-sym}
For a locally free sheaf $B$ on a manifold $X$ one has $$ch(Sym^2 B) = \frac{(ch(B))^2}{2}+r_2 \, \frac{ch(B)}{2}.$$
\end{lem}
Indeed, then $$ch(\wedge^2 ([A]-[B]))=\frac{(ch([A]-[B]))^2}{2}-\frac{r_2ch([A]-[B])}{2},$$
which would prove (\ref{Chern-character-of-wedge-product}) in the general case.
Finally, applying again the Vandermonde determinant argument we obtain the proof of 
Lemma \ref{V-is-Hdg} and, hence,  Lemma \ref{lemma-pushforward} in the general case.
\end{proof}

\begin{section}{Acknowledgements}
I want to sincerely thank my advisor Eyal Markman for introducing me
to the Shafarevich conjecture and to methods of approaching it, and for overall permanent support.
This text is a part of my PhD thesis defended at the University of Massachusetts, Amherst.
I want to sincerely thank the participants of the  Intercity seminar at the University of Amsterdam
in Spring 2016 who  carefully read this text and made many useful suggestions on imroving the exposition.
In particular, I thank Lenny Taelman, Emre Sert\"{o}z and Wessel Bindt.
\end{section}

\vspace*{0.5cm}
Department of Mathematics, University of California San Diego,
nvbuskin@gmail.com

\end{document}